\documentclass[10pt]{article}
\usepackage{mathrsfs}
\usepackage{amssymb}
\usepackage{amsmath}
\usepackage[all]{xy}
\usepackage{amsthm}
\usepackage{amsfonts}
\usepackage{latexsym}
\usepackage{booktabs}
\usepackage{amscd,amsopn,graphics,mathrsfs,accents,enumerate,verbatim,calc}
\usepackage{graphicx}
\usepackage[colorlinks=true,linkcolor=red,citecolor=blue]{hyperref}

\usepackage{stackrel} 

\newtheorem{thm}{Theorem}[section]
\newtheorem{prop}[thm]{Proposition}

\newtheorem{lem}[thm]{Lemma}

\newtheorem{cor}[thm]{Corollary}
\newtheorem{df}[thm]{Definition}
\newtheorem{exa}[thm]{Example}

\newtheorem{rem}[thm]{Remark}

\newtheorem{phy}[thm]{Hypothesis}

\begin{document}

\title{\large \bf Approximability of Triangulated Categories under Separable Extensions
}
\author{Mengzhen Zhu$^a$, Yuefei Zheng$^b$, Xianhui Fu$^{a,}$\footnote{Corresponding author}\\
\it\footnotesize $^a$School of Mathematics and Statistics, Northeast Normal University, Changchun, China\\
\it\footnotesize $^b$College of Science, Northwest A$\&$F University, Yangling 712100, Shaanxi Province, China\\
\it\footnotesize Email addresses: mzzhu@nenu.edu.cn, zhengyf@nwafu.edu.cn, fuxianhui@gmail.com
}
\date{}
\baselineskip=14pt
\maketitle

\begin{abstract}
In this paper, we investigate the behavior of approximability under separable extensions of triangulated categories. Let $\mathcal{B}$ be an approximable triangulated category and let $\mathbf{T}=(T,\theta,\mu)$ be a smashing, stably separable exact monad on $\mathcal{B}$. Assume
that the category of $\mathbf{T}$-modules
$\mathbf{T}\text{-}\mathrm{Mod}_{\mathcal{B}}$ admits a canonical triangulated structure. We prove that if $T$ is right $t$-exact with respect to a $t$-structure in the preferred equivalence class of $\mathcal{B}$, then $\mathbf{T}\text{-}\mathrm{Mod}_{\mathcal{B}}$ is again an approximable triangulated category. As an application, we show that, under suitable assumptions, a smashing localization functor on an approximable triangulated category induces a recollement of approximable triangulated categories, thereby yielding a converse-type result to the gluing theorem of Burke--Neeman--Pauwels. Finally, under suitable hypotheses, we prove that separable extensions restrict naturally to the intrinsic subcategories associated with approximable triangulated categories.
\end{abstract}

\section{Introduction}
Approximable triangulated categories were introduced by Neeman \cite{App, compactge} as a powerful framework to study the triangulated categories arising in algebraic geometry and topology. Roughly speaking, a triangulated category with coproducts is approximable if it is compactly generated by a single object $C$, such that any ``bounded above'' object can be systematically approximated by objects finitely built from coproducts of shifts of $C$. This class encompasses the derived categories of rings and quasi-compact, separated schemes, as well as the homotopy category of spectra. The theory has proved useful in extending the foundational results of Bondal and Van den Bergh \cite{Van} and Rouquier \cite{Dim}; see \cite{str, compactge}. 

The stability of approximability under standard categorical operations has been investigated in several directions. Notably, Burke, Neeman, and Pauwels \cite{glue} demonstrated that approximability behaves well under gluing techniques arising from recollements. Nevertheless, approximability remains a delicate property. Even algebraic triangulated categories possessing a single compact generator, such as the derived category of an arbitrary dg-algebra, can fail to be approximable. If the cohomology vanishes in all positive degrees, the derived category is approximable; see \cite[Remark~5.3]{compactge}.

Parallel to recollements, separable extensions provide another fundamental operation for decomposing and constructing triangulated categories. Introduced by Balmer \cite{septri}, the theory of separable extensions of triangulated categories has become an important tool in tensor-triangular geometry and has subsequently found applications in modular representation theory \cite{stacks, Sepcom} and algebraic geometry \cite{combal}. It encompasses several important constructions, including Bousfield localizations \cite[Example~6.3]{septri}, separated $\acute{\rm e}$tale extensions of schemes \cite[Theorem~3.5]{combal}, and restrictions to subgroups in equivariant stable homotopy categories \cite{sepexample}. Balmer established that, under certain conditions, the Eilenberg--Moore category of modules over a separable exact monad inherits a unique canonical triangulated structure compatible with the original one. 
For convenience, we shall use the following notion of separable extensions, which coincides with the classical one for idempotent complete triangulated categories; see Section~\ref{section 2.3}.

\begin{df}\label{df1}{\rm
A \emph{separable extension} of a triangulated category $\mathcal{B}$ is a triangulated category $\mathcal{A}$ equipped with exact functors $F$ and $G$
\[
\xymatrix{
	 \mathcal{A} \ar[rr]^{G} && \mathcal{B} \ar@/_1.1pc/[ll]_{F},  
}
\]
such that $(F,G)$ is an adjoint pair and $G$ is a separable functor.}
\end{df}

Motivated by these developments, we investigate whether approximability is preserved under separable extensions. Our approach is based on lifting suitable $t$-structures along separable extensions. The key observation is that, if $\mathcal{B}$ is approximable, then lifting a $t$-structure from its preferred equivalence class induces a $t$-structure on $\mathcal{A}$  satisfying Neeman's approximability criteria. This leads to our main theorem, proved in Section~\ref{Section 4}:

\begin{thm}\label{thm:MainA}{\rm
Let $\mathcal{B}$ be an approximable \textup{(respectively, weakly approximable)} triangulated category, and let $\mathcal{A}$ be an idempotent complete separable extension of $\mathcal{B}$. Suppose that $GF$ preserves coproducts and is right $t$-exact with respect to a $t$-structure in the preferred equivalence class of $\mathcal{B}$. Then $\mathcal{A}$ is an approximable \textup{(respectively, weakly approximable)} triangulated category.}
\end{thm}
Since Bousfield localizations form a special class of separable extensions, Theorem~\ref{thm:MainA} applies naturally to smashing localization functors. As an immediate application, we obtain the following recollement result.

\begin{prop}\label{thm:MainB}{\rm
\rm Let $\mathcal{B}$ be an approximable \textup{(respectively, weakly approximable)} triangulated category and $T:\mathcal{B} \to \mathcal{B}$ a smashing localization functor. If the inclusion $\mathrm{Ker}(T)\to \mathcal{B}$ admits a left adjoint, then $T$ induces a recollement
\[
\xymatrix{
	\mathrm{Ker}(T) \ar[rr]^{}  && \mathcal{B} \ar[rr] \ar@/^1.1pc/[ll]\ar@/_1.1pc/[ll]_{} \ && \mathrm{Im}(T) \ar@/_1.1pc/[ll] \ar@/^1.1pc/[ll]
},
\]
where both $\mathrm{Im}(T)$ and $\mathrm{Ker}(T)$ are approximable \textup{(respectively, weakly approximable)} triangulated categories.}
\end{prop}

Proposition \ref{thm:MainB} may be viewed as a converse-type result to the gluing theorem of Burke--Neeman--Pauwels. As a further application, we apply this result to upper triangular matrix dg-algebras. For simplicity, we state only the case where $k$ is a field.

\begin{cor}{\rm
Let
\[
\Lambda=\begin{bmatrix}
  R & M\\
  0 & S
\end{bmatrix}
\]
   be an upper triangular matrix dg-algebra constructed from dg-$k$-algebras $R$ and $S$. If ${\bf D}(\Lambda)$ is approximable, then ${\bf D}(R)$ and ${\bf D}(S)$ are approximable. If moreover $H^{n}(M)=0$ for $n\gg 0$, then ${\bf D}(\Lambda)$ is approximable if and only if ${\bf D}(R)$ and ${\bf D}(S)$ are approximable.}
\end{cor}

Finally, we study the intrinsic subcategories associated with a weakly approximable triangulated category $\mathcal{B}$. Some of these subcategories have also been considered in \cite{KK} and \cite{KPV} in a more general setting, and for completeness we recall their relationships at the beginning of Section~\ref{section5}. We then show that separable extensions restrict naturally to these intrinsic subcategories. We conclude with several examples.

\medskip

\noindent\textbf{Organization of the paper.}
In Section~\ref{section2}, we recall the necessary background on approximable triangulated categories and separable extensions. Section~\ref{Section3} establishes the lifting of $t$-structures to Eilenberg--Moore categories. In Section~\ref{Section 4}, we prove the main theorem on the preservation of approximability under separable extensions and present its applications to recollements and upper triangular matrix dg-algebras. Finally, Section~\ref{section5} studies the restriction of separable extensions to intrinsic subcategories associated with approximable triangulated categories.

\section{Preliminaries}\label{section2}

In this section, we fix notation and recall the background on approximable triangulated categories and separable
extensions needed in the sequel. Our treatment of separable extensions follows Balmer's approach.

\subsection{Notation}
Let $\mathcal{B}=(\mathcal{B},\Sigma,\Omega_{\mathcal{B}})$
be a triangulated category, where $\Sigma$ is the suspension
functor and $\Omega_{\mathcal{B}}$ is the class of distinguished triangles.
For full subcategories $\mathcal{X}$ and $\mathcal{Y}$
of $\mathcal{B}$, let $\mathcal{X}*\mathcal{Y}$ denote
the full subcategory consisting of objects $B$ that fit into a distinguished triangle
$X\longrightarrow B\longrightarrow Y\longrightarrow\Sigma X$
with $X\in\mathcal{X}$ and $Y\in\mathcal{Y}$.
We say that $\mathcal{X}$ is \emph{closed under extensions} if $\mathcal{X}*\mathcal{X}\subseteq\mathcal{X}$.

We write $\mathrm{Hom}_{\mathcal{B}}(\mathcal{X},\mathcal{Y})=0$
if $\mathrm{Hom}_{\mathcal{B}}(X,Y)=0$ for all
$X\in\mathcal{X}$ and $Y\in\mathcal{Y}$.
The right and left orthogonals of $\mathcal{X}$ are the full subcategories 
\[
\begin{aligned}
	\mathcal{X}^{\perp}
	&:=\left\{
	Y\in\mathcal{B}\ \middle|\
	\mathrm{Hom}_{\mathcal{B}}(X,Y)=0
	\text{ for all }X\in\mathcal{X}
	\right\},\\
	{}^{\perp}\mathcal{X}
	&:=\left\{
	Y\in\mathcal{B}\ \middle|\
	\mathrm{Hom}_{\mathcal{B}}(Y,X)=0
	\text{ for all }X\in\mathcal{X}
	\right\}.
\end{aligned}
\]

We recall the following notations from \cite[Reminder~1.1 and Definition 1.3]{str} (see also \cite[Definition 2.1]{BCMPZ}).

\begin{df}\rm
Let $\mathcal{B}$ be a triangulated category with coproducts, and let $\mathcal{X}$ be a class of objects in $\mathcal{B}$.
\begin{enumerate}
    \item[(1)] $\mathrm{smd}(\mathcal{X})$ (resp., $\mathrm{Add}(\mathcal{X})$) denotes the full subcategory of $\mathcal{B}$ consisting of all direct summands (resp., coproducts) of objects in $\mathcal{X}$. 
    \item[(2)] For integers $i\ge 1$, the subcategories $\mathrm{Coprod}_{i}(\mathcal{X})$ of $\mathcal{B}$ are defined inductively by
    \[
    \mathrm{Coprod}_{1}(\mathcal{X}):=\mathrm{Add}(\mathcal{X}),\quad \mathrm{Coprod}_{i+1}(\mathcal{X}):= \mathrm{Coprod}_{1}(\mathcal{X})*\mathrm{Coprod}_{i}(\mathcal{X}).
    \]
    Moreover, $\mathrm{Coprod}(\mathcal{X})$ denotes the smallest full subcategory of $\mathcal{B}$ containing $\mathcal{X}$ that is closed under coproducts and extensions.
    \item[(3)] Let $X$ be an object of $\mathcal{B}$. For integers $m\leq n $, we write $X[m,n]:=\left \{ \Sigma^{-i} X \mid m\le i\le n \right \}$. Similarly, $X(-\infty,n]$, $X[m,+\infty)$, and $X(-\infty,+\infty)$ are defined. We set
    \[
    \overline{\langle X \rangle }_{i}^{[m,n]}:=\mathrm{smd}(\mathrm{Coprod}_{i}(X[m,n])), \quad \overline{\langle X \rangle }^{[m,n]}:=\mathrm{smd}(\mathrm{Coprod}(X[m,n])).
    \]
    We denote by $\overline{\langle X \rangle }$ the smallest full triangulated subcategory of $\mathcal{B}$ containing $X$ and closed under coproducts.
\end{enumerate}
\end{df}

\subsection{Approximable triangulated categories}

We recall Neeman's notion of an approximable triangulated
category \cite{App,compactge}, beginning with the definition
of a $t$-structure \cite{BBD}.

\begin{df}\label{tstructure}\rm
Let $\mathcal{B}$ be a triangulated category. A \emph{$t$-structure} on $\mathcal{B}$ is a pair $(\mathcal{B}^{\le 0},\mathcal{B}^{\ge 0})$ of full subcategories of $\mathcal{B}$ such that:
\begin{enumerate}
    \item[(1)] $\mathcal{B}^{\le 0} \subseteq \mathcal{B}^{\le 1}$ and $\mathcal{B}^{\ge 0} \supseteq  \mathcal{B}^{\ge 1}$, where $\mathcal{B}^{\le 1}:=\Sigma^{-1}\mathcal{B}^{\le 0}$ and $\mathcal{B}^{\ge 1}:=\Sigma^{-1}\mathcal{B}^{\ge 0}$.
    \item[(2)] $\mathrm{Hom}_{\mathcal{B}}(X,Y)=0$ for any $X$ in $\mathcal{B}^{\le 0}$ and any $Y$ in $\mathcal{B}^{\ge 1}$.
    \item[(3)] For any object $B \in \mathcal{B}$, there exists a distinguished triangle
    \[
    X\longrightarrow B\longrightarrow Y\longrightarrow \Sigma X
    \]
    in $\mathcal{B}$, with $X\in \mathcal{B}^{\le 0}$ and $Y\in \mathcal{B}^{\ge 1}$.
\end{enumerate}
\end{df}

Let $(\mathcal{B}^{\le 0},\mathcal{B}^{\ge 0})$ be a $t$-structure on $\mathcal{B}$. The category $\mathcal{H}_{\mathcal{B}}:=\mathcal{B}^{\le 0} \cap \mathcal{B}^{\ge 0}$ is abelian and is called the \emph{heart} of the $t$-structure. We set
\[
\mathcal{B}^{-}:=  \bigcup_{n>0 } \mathcal{B}^{\le n}, \quad \mathcal{B}^{+}:= \bigcup_{n>0} \mathcal{B}^{\ge -n} \quad \text{and} \quad \mathcal{B}^{b}:=\mathcal{B}^{-}\bigcap \mathcal{B}^{+},
\]
where $\mathcal{B}^{\le n}:=\Sigma^{-n}\mathcal{B}^{\le 0}$ and $\mathcal{B}^{\ge n}:=\Sigma^{-n}\mathcal{B}^{\ge 0}$. We say that $(\mathcal{B}^{\le 0},\mathcal{B}^{\ge 0})$ is \emph{bounded above}, \emph{bounded below}, and \emph{bounded} if $\mathcal{B}=\mathcal{B}^{-}$, $\mathcal{B}=\mathcal{B}^{+}$, and $\mathcal{B}=\mathcal{B}^{b}$, respectively. The $t$-structure $(\mathcal{B}^{\le 0},\mathcal{B}^{\ge 0})$ is called \emph{nondegenerate} if $\bigcap_{n\in \mathbb{N} }\mathcal{B}^{\le -n}=\bigcap_{n\in \mathbb{N} }\mathcal{B}^{\ge n}=0$.

Two $t$-structures $(\mathcal{B}^{\le 0}_{1},\mathcal{B}^{\ge 0}_{1})$ and $(\mathcal{B}^{\le 0}_{2},\mathcal{B}^{\ge 0}_{2})$ are said to be \emph{equivalent} if there exists a natural number $n$ such that $\mathcal{B}^{\le -n}_{1}\subseteq \mathcal{B}^{\le 0}_{2}\subseteq \mathcal{B}^{\le n}_{1}$. Clearly, equivalent $t$-structures give rise to identical $\mathcal{B}^{-}$, $\mathcal{B}^{+}$, and $\mathcal{B}^{b}$; see \cite[Observation 1.16]{compactge}.

\begin{df}\rm
Let $\mathcal{B}$ be a triangulated category with coproducts. An object $C\in \mathcal{B}$ is \emph{compact} if the functor $\mathrm{Hom}_{\mathcal{B}}(C,-)$ preserves coproducts. A compact object \(C\) is called a {\it compact generator} for \(\mathcal B\) if for any $X\in \mathcal{B}$, $\mathrm{Hom}_{\mathcal{B}}(C,\Sigma^{n}X)=0$ for all $n \in \mathbb{Z} $ implies $X=0$.
\end{df}

We denote by $\mathcal{B}^{c}$ the full subcategory of $\mathcal{B}$ consisting of all compact objects. This is a full triangulated subcategory of $\mathcal{B}$ closed under direct summands. For any compact object $C$ of $\mathcal{B}$, there exists a unique $t$-structure $(\mathcal{B}^{\le 0}_{C},\mathcal{B}^{\ge 0}_{C}):=(\overline{\langle C \rangle }^{(-\infty,0]},C(-\infty,-1]^{\perp})$ on $\mathcal{B}$, which is called \emph{the $t$-structure generated by} $C$; see \cite[Theorem A.1 and  Proposition A.2]{Construction}.

\begin{df}\rm (\cite[Definition~1.18]{compactge})
Let $\mathcal{B}$ be a compactly generated triangulated category with a compact generator $C$. Then the \emph{preferred equivalence class of $t$-structures} on $\mathcal{B}$ is defined to be the one containing the $t$-structure $(\mathcal{B}^{\le 0}_{C},\mathcal{B}^{\ge 0}_{C})$ generated by $C$.
\end{df}

We now recall the definitions of weakly approximable and approximable triangulated categories.

\begin{df}\label{approximable}\rm (\cite[Definition~1.25]{compactge})
A triangulated category $\mathcal{B}$ with coproducts is \emph{weakly approximable} if there exists a compact generator $C$, a $t$-structure $(\mathcal{B}^{\le 0},\mathcal{B}^{\ge 0})$, and an integer $n>0$, such that the following hold:
\begin{enumerate}
    \item[(1)] $\Sigma^{n} C\in \mathcal{B}^{\le 0}$ and $\mathrm{Hom}_{\mathcal{B}}(\Sigma^{-n} C,\mathcal{B}^{\le 0})=0$.
    \item[(2)] For every $B \in \mathcal{B}^{\le 0}$, there exists a distinguished triangle
    \[
    X\longrightarrow B\longrightarrow Y\longrightarrow \Sigma X
    \]
    in $\mathcal{B}$, with $X\in \overline{\langle C \rangle} ^{[-n,n]}$ and $Y\in \mathcal{B}^{\le -1}$.
\end{enumerate}
The category $\mathcal{B}$ is called \emph{approximable} if the integer $n$ can be chosen to further satisfy:
\begin{enumerate}
    \item[(3)] For every $B \in \mathcal{B}^{\le 0}$, there exists a distinguished triangle
    \[
    X\longrightarrow B\longrightarrow Y\longrightarrow \Sigma X
    \]
    in $\mathcal{B}$, with $X\in \overline{\langle C \rangle}_{n} ^{[-n,n]}$ and $Y\in \mathcal{B}^{\le -1}$.
\end{enumerate}
\end{df}

By \cite[Proposition~3.4]{compactge}, the $t$-structure appearing in Definition~\ref{approximable} necessarily
belongs to the preferred equivalence class. Moreover, if $\mathcal{B}$ is weakly approximable, then for any compact generator $C$ and any $t$-structure in the preferred equivalence class, there exists an integer $n>0$ for which conditions \textup{(1)} and \textup{(2)} of Definition~\ref{approximable} hold.
If $\mathcal{B}$ is approximable, $n$ can also be chosen to satisfy \textup{(3)}; see
\cite[Proposition~3.6]{compactge}.
The integer $n$ may depend on the chosen generator and $t$-structure, but is independent of $B\in\mathcal{B}^{\le 0}$.

\subsection{Separable extensions of triangulated categories}\label{section 2.3}

In this subsection, we recall some basic facts from \cite{septri} concerning separable extensions of triangulated categories, which build upon the theory of monads; see \cite{barr2000toposes, categories}.

Let $\mathcal{B}$ be an additive category. A \emph{monad} on $\mathcal{B}$ is a triple ${\bf T}= (T,\theta, \mu)$ consisting of an additive endofunctor $T:\mathcal{B} \to \mathcal{B}$ and two natural transformations $\theta:\mathrm{Id}_{\mathcal{B} }\to T$ (the unit) and $\mu :T^{2} \to T$ (the multiplication), subject to the relations $\mu \circ T\mu =\mu \circ \mu T $ and $\mu \circ T\theta =\mu \circ \theta T=\mathrm{Id}_{T } $. Let $\mathcal{A}$ and $\mathcal{B}$ be additive categories, and consider an adjunction
\[
\xymatrix{
     \mathcal{A} \ar[rr]^{G} && \mathcal{B} \ar@/_1.1pc/[ll]_{F} 
}
\]
from $\mathcal{B}$ to $\mathcal{A}$ with unit $\theta:\mathrm{Id}_{\mathcal{B} }\to GF$ and counit $\varepsilon : FG\to \mathrm{Id}_{\mathcal{A}}$. This adjunction $(F, G; \theta, \varepsilon)$ defines a monad ${\bf T}= (T=GF,\theta, \mu=G\varepsilon F)$ on $\mathcal{B}$, which is called the \emph{monad defined by the adjunction}.

Conversely, every monad arises from an adjunction,
as shown by the constructions of Eilenberg--Moore
\cite{EM} and Kleisli \cite{Kleisli}. We recall the Eilenberg--Moore construction. For a monad ${\bf T}$ on $\mathcal{B}$, a \emph{${\bf T}$-module} is a pair $(B, h_{B})$ consisting of an object $B\in \mathcal{B}$ (the underlying object) and a morphism $h_{B}:T(B)\to B$ in $\mathcal{B}$, subject to the relations $h_{B}\circ \theta _{B}=\mathrm{Id}_{B}$ and $h_{B}\circ T(h_{B})=h_{B}\circ \mu _{B}$. A morphism $f:(B, h_{B})\to(B^{\prime}, h_{B^{\prime}})$ of ${\bf T}$-modules is a morphism $f:B\to B^{\prime}$ in $\mathcal{B}$ satisfying $h_{B^{\prime }}\circ T(f)=f\circ h_{B}$. This defines the \emph{Eilenberg--Moore category} ${\bf T}\text{-}\mathrm{Mod}_{\mathcal{B}}$ of ${\bf T}$-modules. 

Consider the functor $F^{{\bf T}}:\mathcal{B} \to {\bf T}\text{-}\mathrm{Mod}_{\mathcal{B}}$ sending $B$ to $(T(B),\mu_{B})$ and a morphism $f:A\to B$ to the morphism $T(f)$. Denote by $G^{{\bf T}}:{\bf T}\text{-}\mathrm{Mod}_{\mathcal{B}} \to \mathcal{B}$ the forgetful functor. Then $(F^{{\bf T}}, G^{{\bf T}}; \theta^{{\bf T}}, \varepsilon^{{\bf T}})$ forms an adjunction from $\mathcal{B}$ to ${\bf T}\text{-}\mathrm{Mod}_{\mathcal{B}}$ with $G^{\bf T}F^{\bf T}=T$. Note that $\theta ^{{\bf T}}=\theta$ and $\varepsilon^{{\bf T}}(B, h_{B})=h_{B}$ for each ${\bf T}$-module $(B, h_{B})$.

In light of the preceding discussion, an adjunction $(F, G; \theta, \varepsilon)$ induces a monad ${\bf T}= (GF,\theta, G\varepsilon F)$ on $\mathcal{B}$. By \cite[Theorem VI.3.1]{categories}, there exists a unique functor $\varphi: \mathcal{A} \to {\bf T}\text{-}\mathrm{Mod}_{\mathcal{B}}$ that maps an object $A$ to $(G(A), G(\varepsilon_{A} ))$ and a morphism $f$ to $G(f)$, making the following diagram commutative:
\[
\xymatrix@=1.2cm{
A \ar[d]^{\varphi } \ar[r]^{G} & B \ar@/_1pc/[l]_{F} \ar@{=}[d]\\
{\bf T}\text{-}\mathrm{Mod}_{\mathcal{B}} \ar[r]^-{G^{\bf T}} & B. \ar@/_1pc/[l]_-{F^{\bf T}}
}
\]
The unique functor $\varphi$ is called the \emph{Eilenberg--Moore comparison functor}. If $\varphi$ is an equivalence of categories, then $G$ is said to be \emph{monadic}.

It is a standard fact that ${\bf T}\text{-}\mathrm{Mod}_{\mathcal{B}}$ is an additive category. Referring to \cite[Proposition 4.3.2]{handbook}, we have the following fact which states that coproducts in $\mathcal{B}$ induce  coproducts in ${\bf T}\text{-}\mathrm{Mod}_{\mathcal{B}}$. 

\begin{lem}\label{coproducts}\rm
Let $\mathcal{B}$ be an additive category with coproducts and ${\bf T}=(T, \theta, \mu)$ a monad on $\mathcal{B}$. Suppose that $T$ preserves coproducts. Then ${\bf T}\text{-}\mathrm{Mod}_{\mathcal{B}}$ has coproducts, and the forgetful functor $G^{{\bf T}}$ preserves them.
\end{lem}	
\begin{proof}\rm
Let $\{(X_\lambda,h_{X_\lambda})\}_{\lambda\in\Lambda}$
be a family of $\mathbf{T}$-modules indexed by a set
$\Lambda$. The underlying objects $\left \{ X_{\lambda } \right \}_{\lambda\in \Lambda} $ form a set in $\mathcal{B}$. Since $\mathcal{B}$ has coproducts, let $(\coprod_{\lambda \in \Lambda }X_{\lambda},e_{\lambda})$ denote the coproduct of $\left \{ X_{\lambda } \right \}_{\lambda\in \Lambda} $, where $e_{\lambda}:X_{\lambda}\to \coprod_{\lambda \in \Lambda }X_{\lambda}$ are the canonical injections. Since $T$ preserves coproducts, we obtain a unique factorization $h:T(\coprod_{\lambda \in \Lambda }X_{\lambda})\to \coprod_{\lambda \in \Lambda }X_{\lambda}$ such that the following diagram is commutative for any $\lambda \in \Lambda$:
\[
\xymatrix{
    T(X_{\lambda}) \ar[rr]^{T(e_{\lambda})} \ar[d]^{h_{X_{\lambda}}} &&  T(\coprod_{\lambda \in \Lambda }X_{\lambda}) \ar[d]^{h} \\
    X_{\lambda} \ar[rr]^{e_{\lambda}} && \coprod_{\lambda \in \Lambda }X_{\lambda}.
}
\]
We claim that $(\coprod_{\lambda \in \Lambda }X_{\lambda},h)$ is a ${\bf T}$-module. First, since $(X_{\lambda},h_{X_\lambda})$ is a ${\bf T}$-module, we have $h_{X_{\lambda}} \circ \theta_{X_{\lambda}}=\mathrm{Id}_{X_\lambda} $. The natural transformation $\theta$ induces the following commutative diagram for any $\lambda \in \Lambda$:
\[
\xymatrix{
    T(X_{\lambda}) \ar[rr]^{T(e_{\lambda})}  &&  T(\coprod_{\lambda \in \Lambda }X_{\lambda}) \\
    X_{\lambda} \ar[rr]^{e_{\lambda}} \ar[u]^{\theta _{X_{\lambda}}} && \coprod_{\lambda \in \Lambda }X_{\lambda}. \ar[u]^{\theta _{\coprod_{\lambda \in \Lambda }X_{\lambda}}}
}
\]
Hence, $e_{\lambda} = e_{\lambda}\circ h_{X_{\lambda}} \circ \theta_{X_{\lambda}} = h \circ \theta_{\coprod_{\lambda \in \Lambda }X_{\lambda}} \circ e_{\lambda}$. Because the coproduct is equipped with a universal property, this implies $h \circ \theta_{\coprod_{\lambda \in \Lambda }X_{\lambda}}=\mathrm{Id}_{\coprod_{\lambda \in \Lambda }X_{\lambda}}$.

Next, since $T$ preserves coproducts, $\left(T^{2}(\coprod_{\lambda \in \Lambda }X_{\lambda}),T^{2}(e_{\lambda})\right)$ is the coproduct of $\left \{ T^{2}(X_{\lambda }) \right \}_{\lambda\in \Lambda}$. The natural transformation $\mu$ induces the following commutative diagram for any $\lambda \in \Lambda$:
\[
\xymatrix{
     T^{2}(X_{\lambda })  \ar[rr]^{T^{2}(e_{\lambda})} \ar[d]^{\mu _{X_{\lambda}}} &&  T^{2}(\coprod_{\lambda \in \Lambda }X_{\lambda}) \ar[d]^{\mu_{\coprod_{\lambda \in \Lambda }X_{\lambda}}} \\
    T(X_{\lambda}) \ar[rr]^{T(e_{\lambda})} && T(\coprod_{\lambda \in \Lambda }X_{\lambda}).
}
\]
We then have the following sequence of equalities:
\begin{align*}
    h \circ \mu_{\coprod_{\lambda \in \Lambda }X_{\lambda}} \circ T^{2}(e_{\lambda}) 
    &= h \circ T(e_{\lambda}) \circ \mu_{X_{\lambda}} \\
    &= e_{\lambda} \circ h_{X_{\lambda}} \circ \mu_{X_{\lambda}} \\
    &= e_{\lambda} \circ  h_{X_{\lambda}} \circ T(h_{X_{\lambda}}) \\
    &= h \circ T(e_{\lambda}) \circ T(h_{X_{\lambda}}) \\
    &= h \circ T(h) \circ T^{2}(e_{\lambda}).
\end{align*}
This implies $h \circ \mu_{\coprod_{\lambda \in \Lambda }X_{\lambda}} = h \circ T(h)$. Thus, $(\coprod_{\lambda \in \Lambda }X_{\lambda},h)$ is indeed a ${\bf T}$-module. Furthermore, the relation $h\circ T(e_{\lambda})=e_{\lambda}\circ h_{X_{\lambda}}$ shows that the injections $e_{\lambda}:(X_{\lambda },h_{X_{\lambda}})\to (\coprod_{\lambda \in \Lambda }X_{\lambda},h)$ are morphisms of ${\bf T}$-modules for all $\lambda \in \Lambda$.

Finally, given an object $(Y,h_{Y})$ and morphisms $f_{\lambda}:(X_{\lambda},h_{X_{\lambda}})\to (Y,h_{Y})$ in ${\bf T}\text{-}\mathrm{Mod}_{\mathcal{B}}$, there exists a unique factorization $s:\coprod_{\lambda \in \Lambda }X_{\lambda}\to Y$ in $\mathcal{B}$ such that $f_{\lambda}=s\circ e_{\lambda}$ for all $\lambda \in \Lambda$. We deduce that $s:(\coprod_{\lambda \in \Lambda }X_{\lambda},h)\to (Y,h_{Y})$ is actually a morphism in ${\bf T}\text{-}\mathrm{Mod}_{\mathcal{B}}$, because
\[
h_{Y}\circ T(s)\circ T(e_{\lambda}) = h_{Y}\circ T(f_{\lambda}) = f_{\lambda}\circ h_{X_{\lambda}} = s\circ e_{\lambda} \circ h_{X_{\lambda}} = s\circ h\circ T(e_{\lambda}),
\]
where we used the assumption that $T$ preserves coproducts. This proves that $(\coprod_{\lambda \in \Lambda }X_{\lambda},h)$ is the coproduct of $\left \{ (X_{\lambda },h_{X_{\lambda}}) \right \}_{\lambda\in \Lambda}$. It is clear from the construction that the forgetful functor $G^{\bf T}$ preserves coproducts.
\end{proof}

\begin{rem}\rm 
Under the assumptions of Lemma~\ref{coproducts},
the free module functor $F^{\mathbf{T}}$ also preserves coproducts, since it is left adjoint to $G^{\mathbf{T}}$. Thus both functors in the Eilenberg--Moore adjunction preserve coproducts.
\end{rem}

Since the underlying functor $T$ is additive,
idempotent completeness of $\mathcal{B}$ implies
idempotent completeness of
$\mathbf{T}\text{-}\mathrm{Mod}_{\mathcal{B}}$;
see \cite[Remark~2.6]{septri}.
When the category $\mathcal{B}$ is equipped with a suspension $\Sigma$, an additive monad ${\bf T}$ is said to be \emph{stable} if $T$, $\theta$, and $\mu$ commute with the suspension: $T\Sigma =\Sigma T$, $\theta \Sigma =\Sigma \theta$, and $\mu \Sigma =\Sigma \mu$. If moreover, $\mathcal{A}$ and $\mathcal{B}$ are suspended and $F$ and $G$ commute with the suspension, we assume that the adjunction respects the suspension, meaning that $\theta$ and $\varepsilon$ commute with the suspension. In this case, the monad ${\bf T}= (GF,\theta, G\varepsilon F)$ is inherently stable.

Unless otherwise stated, we henceforth assume that $\mathcal{B}$ is a triangulated category. An additive functor between triangulated categories
is called \emph{exact} if it is equipped with a natural isomorphism commuting with suspension and sends distinguished triangles to distinguished triangles using this isomorphism.
A monad $\mathbf{T}$ on $\mathcal{B}$ is called \emph{exact} if it is stable and its underlying functor $T$ is exact. Recall from \cite[Proposition~6.3]{hopf} that a monad ${\bf T}= (T,\theta, \mu)$ on a category $\mathcal{B}$ is \emph{separable} provided that there exists a natural transformation $\varpi :T\to T^{2}$ satisfying $\mu \circ \varpi =\mathrm{Id}_{T}$ and $T\mu \circ \varpi T=\varpi \circ \mu=\mu T\circ T\varpi$. If $\varpi :T\to T^{2}$ commutes with the suspension as well, we say that ${\bf T}$ is \emph{stably separable}. The following result is implicit in \cite[Theorem~5.17]{septri}; see also \cite[Lemma 2.4]{KK}, where the authors drop the idempotent completeness assumption.

\begin{thm}\label{balmer}\rm
	Let $\mathcal{B}$ be a triangulated category, and let
	$\mathbf{T}=(T,\theta,\mu)$ be a stably separable exact
	monad on $\mathcal{B}$.
	Then the Eilenberg--Moore category
	$\mathbf{T}\text{-}\mathrm{Mod}_{\mathcal{B}}$
	admits a unique pre-triangulated structure for which
	both $F^{\mathbf{T}}$ and $G^{\mathbf{T}}$ are exact.
	A triangle $\Delta$ of $\mathbf{T}$-modules is
	distinguished if and only if $G^{\mathbf{T}}(\Delta)$
	is distinguished in $\mathcal{B}$.
	We call this the \emph{canonical pre-triangulated structure}.
\end{thm}

When this canonical structure is triangulated,
$\mathbf{T}\text{-}\mathrm{Mod}_{\mathcal{B}}$
is called a \emph{separable extension} of $\mathcal{B}$. Examples include:
\begin{enumerate}
    \item[(1)] Bousfield localizations, see \cite[Example~6.3]{septri};
    \item[(2)] \'Etale extensions in algebraic geometry, see \cite[Theorem~3.5]{combal};
    \item[(3)] Restriction to subgroups in equivariant stable homotopy categories, in equivariant KK-theory, and in equivariant derived categories, see \cite{sepexample}.
\end{enumerate}

\begin{rem}\label{remsep}\rm
\begin{enumerate}
    \item[(1)] It is well known that if we assume $\mathcal{B}$ admits a triangulation of order $N\geq 3$, then ${\bf T}\text{-}\mathrm{Mod}_{\mathcal{B}}$ is indeed a valid triangulated category \cite[Section 5]{septri}. In the remainder of this paper, we assume that the canonical pre-triangulated structure on
    $\mathbf{T}\text{-}\mathrm{Mod}_{\mathcal{B}}$
    is triangulated.
    \item[(2)]
    For any monad $\mathbf{T}$ on $\mathcal{B}$,
    $\mathbf{T}$ is separable if and only if
    $G^{\mathbf{T}}$ is a separable functor;
    see \cite[Lemma~3.2]{Anote}.
    Equivalently, the counit
    $\varepsilon^{\mathbf{T}}$ admits a natural section;
    see \cite[Lemma~2.2]{Anote}.
    \item[(3)] 
    Let $(F, G; \theta, \varepsilon)$ be an adjunction from $\mathcal{B}$ to $\mathcal{A}$ with $\mathcal{A}$ idempotent complete.
    For the induced monad
    $\mathbf{T}=(GF,\theta,G\varepsilon F)$,
    the functor $G$ is separable if and only if
    $\mathbf{T}$ is separable and the comparison functor
    $\varphi:\mathcal{A}\longrightarrow
    \mathbf{T}\text{-}\mathrm{Mod}_{\mathcal{B}}$
    is an equivalence; see
    \cite[Corollary~3.6]{Anote}.
\end{enumerate}
\end{rem}

Based on the descriptions above, we make the following hypothesis for the remainder of this paper.

\begin{phy}\label{ph}\rm
Let $\mathcal{B}$ be a triangulated category and let ${\bf T}=(T,\theta,\mu)$ be a stably separable exact monad on $\mathcal{B}$. Assume that the canonical pre-triangulated structure on
$\mathbf{T}\text{-}\mathrm{Mod}_{\mathcal{B}}$ is triangulated. Whenever $\mathcal{B}$ admits coproducts, we further assume that $\mathbf{T}$ is \emph{smashing}, meaning that $T$ preserves coproducts.
\end{phy}

\section{Lifting t-structures}\label{Section3}

In this section, we investigate the lifting of $t$-structures from $\mathcal{B}$ to the Eilenberg--Moore category $\mathbf{T}\text{-}\mathrm{Mod}_{\mathcal{B}}$, with particular attention to $t$-structures generated by compact objects. We begin by recalling the following terminology.

\begin{df}\rm
Let $\mathcal{A}$ and $\mathcal{B}$ be triangulated categories equipped with $t$-structures $(\mathcal{A}^{\le 0},\mathcal{A}^{\ge 0})$ and $(\mathcal{B}^{\le 0},\mathcal{B}^{\ge 0})$, respectively. An exact functor $G:\mathcal{A}\to \mathcal{B}$ is said to be:
\begin{enumerate}
    \item[(1)] \emph{left $t$-exact} if $G(\mathcal{A}^{\ge 0})\subseteq \mathcal{B}^{\ge 0}$;
    \item[(2)] \emph{right $t$-exact} if $G(\mathcal{A}^{\le 0})\subseteq \mathcal{B}^{\le 0}$;
    \item[(3)] \emph{$t$-exact} if it is both left and right $t$-exact.
\end{enumerate}
\end{df}

For a full subcategory $\mathcal{S}$ of $\mathcal{B}$, let $\mathcal{S}^{\mathbf{T}}$ denote the full subcategory of $\mathbf{T}\text{-}\mathrm{Mod}_{\mathcal{B}}$ consisting of the modules $(X,h_{X}) \in \mathbf{T}\text{-}\mathrm{Mod}_{\mathcal{B}}$ whose underlying objects $X$ belong to $\mathcal{S}$. This notation does not require $T(\mathcal{S}) \subseteq \mathcal{S}$. We record the following elementary observation for later use.

\begin{lem}\label{stab}\rm
	Assume that $\mathcal{B}$ satisfies
	Hypothesis~\ref{ph}. Let $\mathcal{S}$ be a triangulated
	(resp., thick) subcategory of $\mathcal{B}$. Then
	$\mathcal{S}^{\mathbf{T}}$ is a triangulated
	(resp., thick) subcategory of
	$\mathbf{T}\text{-}\mathrm{Mod}_{\mathcal{B}}$. In particular, if $T(\mathcal{S})\subseteq\mathcal{S}$,
	then $\mathbf{T}$ restricts to a stably separable exact monad $\mathbf{T}|_{\mathcal{S}}$ on
	$\mathcal{S}$, and
	$\mathcal{S}^{\mathbf{T}}
	=
	(\mathbf{T}|_{\mathcal{S}})\text{-}\mathrm{Mod}_{\mathcal{S}}$
	is a separable extension of $\mathcal{S}$.
\end{lem}

\begin{proof}
	By Theorem~\ref{balmer},
	the forgetful functor $G^{\mathbf{T}}:\mathbf{T}\text{-}\mathrm{Mod}_{\mathcal{B}} \to \mathcal{B}$ is exact. Hence its inverse image 
	$\mathcal{S}^{\mathbf{T}}$ is closed under shifts and cones
	whenever $\mathcal{S}$ is a triangulated subcategory. Thus  $\mathcal{S}^{\mathbf{T}}$ is a
	triangulated subcategory. If $\mathcal{S}$ is a thick subcategory and $(X,h_X)$ is a direct summand
	of $(Y,h_Y)\in\mathcal{S}^{\mathbf{T}}$, then $X$ is a
	direct summand of $Y$ in $\mathcal{B}$. Thus $X\in\mathcal{S}$,
	and therefore $(X,h_X)\in\mathcal{S}^{\mathbf{T}}$. Consequently,
	$\mathcal{S}^{\mathbf{T}}$ is thick.
	
	Now suppose that $T(\mathcal{S})\subseteq\mathcal{S}$.
	Since $\mathcal{S}$ is full, the unit, multiplication, and
	separability section of $\mathbf{T}$ restrict to
	$\mathcal{S}$, together with their compatibility with the
	suspension. The restriction of $T$ is exact, so
	$\mathbf{T}|_{\mathcal{S}}$ is a stably separable exact
	monad on $\mathcal{S}$. By definition, modules over this restricted monad are
	precisely the $\mathbf{T}$-modules whose underlying objects
	belong to $\mathcal{S}$. Hence
	$\mathcal{S}^{\mathbf{T}}
	= (\mathbf{T}|_{\mathcal{S}})
	\text{-}\mathrm{Mod}_{\mathcal{S}}$
	is a separable extension of $\mathcal{S}$.
\end{proof}
The restriction considered above is closely related to
\cite[Lemma~2.7]{KK}. More precisely, suppose that the
Eilenberg--Moore adjunction restricts to an adjunction between full triangulated subcategories $\mathcal{S}$ and $\mathcal{V}$ of $\mathcal{B}$ and
$\mathbf{T}\text{-}\mathrm{Mod}_{\mathcal{B}}$, respectively, and $\mathcal{V}$ is thick. Then \cite[Lemma~2.7]{KK} gives
$\mathcal{V}=
\mathcal{S}^{\mathbf{T}} =(\mathbf{T}|_{\mathcal{S}})
\text{-}\mathrm{Mod}_{\mathcal{S}}$. When no confusion can arise, we write $\mathbf{T} \text{-}\mathrm{Mod}_{\mathcal{S}}$ for $(\mathbf{T}|_{\mathcal{S}})
\text{-}\mathrm{Mod}_{\mathcal{S}}$.

The following theorem is the main result of this section. It shows that a $t$-structure on $\mathcal{B}$ lifts to
$\mathbf{T}\text{-}\mathrm{Mod}_{\mathcal{B}}$ whenever
$T$ is right $t$-exact.

\begin{thm}\label{tot}\rm
Assume that $\mathcal{B}$ satisfies Hypothesis~\ref{ph}. Let $(\mathcal{B}^{\le 0},\mathcal{B}^{\ge 0})$ be a $t$-structure on $\mathcal{B}$. If $T$ is right $t$-exact with respect to $(\mathcal{B}^{\le 0},\mathcal{B}^{\ge 0})$, then $\left (\mathbf{T}\text{-}\mathrm{Mod}_{\mathcal{B}^{\le 0}},(\mathcal{B}^{\ge 0})^{\mathbf{T}}\right)$ is a $t$-structure on $\mathbf{T}\text{-}\mathrm{Mod}_{\mathcal{B}}$.
\end{thm}

\begin{proof}
    By Theorem~\ref{balmer} and Hypothesis~\ref{ph},
	$\mathbf{T}\text{-}\mathrm{Mod}_{\mathcal{B}}$ has the canonical triangulated structure,
	and the forgetful functor $G^{\mathbf{T}}$ is exact. Since $T(\mathcal{B}^{\le 0})\subseteq\mathcal{B}^{\le 0}$,
	the monad $\mathbf{T}$ restricts to $\mathcal{B}^{\le 0}$,
	and hence $\mathbf{T}\text{-}\mathrm{Mod}_{\mathcal{B}^{\le 0}}=(\mathcal{B}^{\le 0})^{\mathbf{T}}$. We first show that the pair $\left (\mathbf{T}\text{-}\mathrm{Mod}_{\mathcal{B}^{\le 0}},(\mathcal{B}^{\ge 0})^{\mathbf{T}}\right )$ satisfies conditions (1) and (2) of Definition~\ref{tstructure}.
	
The shift conditions
\[
\mathbf{T}\text{-}\mathrm{Mod}_{\mathcal{B}^{\le 0}} \subseteq \mathbf{T}\text{-}\mathrm{Mod}_{\mathcal{B}^{\le 1}}=\Sigma^{-1}\mathbf{T}\text{-}\mathrm{Mod}_{\mathcal{B}^{\le 0}},
\qquad
\Sigma^{-1}(\mathcal{B}^{\ge 0})^{\mathbf{T}}=(\mathcal{B}^{\ge 1})^{\mathbf{T}}\subseteq (\mathcal{B}^{\ge 0})^{\mathbf{T}}
\]
follow from the corresponding conditions on $\mathcal{B}$.
Moreover, if
$(X,h_{X})\in \mathbf{T}\text{-}\mathrm{Mod}_{\mathcal{B}^{\le 0}}$ and $(Y,h_{Y})\in \Sigma^{-1}(\mathcal{B}^{\ge 0})^{\mathbf{T}}$, then
$X\in\mathcal{B}^{\le 0}$ and
$Y\in\mathcal{B}^{\ge 1}$.
Since $G^{\mathbf{T}}$ is faithful, the vanishing
$\mathrm{Hom}_{\mathcal{B}}(X,Y)=0$ implies
$\mathrm{Hom}_{\mathbf{T}\text{-}\mathrm{Mod}_{\mathcal{B}}}
\bigl((X,h_X),(Y,h_Y)\bigr)=0$.
 Thus, conditions (1) and (2) hold.

Next, we verify condition (3) in Definition~\ref{tstructure}. Let $(B,h_{B})\in \mathbf{T}\text{-}\mathrm{Mod}_{\mathcal{B}}$. Since $\mathcal{B}$ is equipped with a $t$-structure, there exists a distinguished triangle
\[
X\overset{f}\longrightarrow B \overset{g}{\longrightarrow}Y\overset{h}{\longrightarrow} \Sigma X,
\]
in $\mathcal{B}$ with $X\in \mathcal{B}^{\le 0}$ and $Y\in \mathcal{B}^{\ge 1}$. Because $T$ is right $t$-exact, we have $T(X)\in \mathcal{B}^{\le 0}$ and hence $\mathrm{Hom}_{\mathcal{B}}(T(X),Y) =0$. Thus, there exist morphisms $h_X$ and $h_Y$ such that the following diagram commutes:
\[
\xymatrix{
	T(X) \ar[r]^{T(f)}\ar[d]^{h_{X}} & T(B) \ar[d]^{h_{B}} \ar[r]^{T(g)} & T(Y) \ar[d]^{h_{Y}} \ar[r]^{T(h)} & \Sigma T(X) \ar[d]^{\Sigma(h_{X})}  \\
    X \ar[r]^{f} & B \ar[r]^{g} & Y \ar[r]^{h} & \Sigma X,
}
\]
where both rows are distinguished triangles. We claim that this diagram yields the required distinguished triangle for condition (3). To see this, we only need to show that $(X,h_{X})$ and $(Y,h_{Y})$ are well-defined $\mathbf{T}$-modules.

By applying the exact functor $T$ to the diagram above, we obtain the following commutative diagram:
\[
\xymatrix{
	T^{2}(X) \ar[r]^{T^{2}(f)}\ar[d]^{T(h_{X})} & T^{2}(B) \ar[d]^{T(h_{B})} \ar[r]^{T^{2}(g)} & T^{2}(Y) \ar[d]^{T(h_{Y})} \ar[r]^{T^{2}(h)} & \Sigma T^2(X) \ar[d]^{\Sigma T(h_{X})}  \\
	T(X) \ar[r]^{T(f)}\ar[d]^{h_{X}} & T(B) \ar[d]^{h_{B}} \ar[r]^{T(g)} & T(Y) \ar[d]^{h_{Y}} \ar[r]^{T(h)} & \Sigma T(X) \ar[d]^{\Sigma(h_{X})}  \\
    X \ar[r]^{f} & B \ar[r]^{g} & Y \ar[r]^{h} & \Sigma X.
}
\]
Since $T$ is right $t$-exact, it is obvious that $\mathrm{Hom}_{\mathcal{B}}(T^{2}(X),\Sigma^{-1} Y)=0$. By \cite[Chapter~2, Proposition~1.4.5]{derived}, $h_{X}\circ T(h_{X})$ and $h_{Y}\circ T(h_{Y})$ are the unique morphisms such that the left and right outer squares are commutative, respectively. We then have the sequence of equalities:
\begin{align*}
h_{B}\circ T(h_{B}) \circ T^{2}(f) &= h_{B}\circ \mu_{B}\circ T^{2}(f) \\
&= h_{B}\circ T(f)\circ \mu_{X} \\
&= f\circ h_{X}\circ \mu_{X}.
\end{align*}
By the uniqueness of $h_{X}\circ T(h_{X})$, we obtain $h_{X}\circ T(h_{X})=h_{X}\circ \mu_{X}$. A similar argument yields $h_{Y}\circ T(h_{Y})=h_{Y}\circ \mu_{Y}$.

Consider the following commutative diagram, where the vertical morphisms in the upper part of the diagram are induced by the natural transformation $\theta$:
\[
\xymatrix{
	X \ar[r]^{f}\ar[d]^{\theta_{X}} & B \ar[d]^{\theta_{B}} \ar[r]^{g} & Y \ar[d]^{\theta_{Y}} \ar[r]^{h} & \Sigma X \ar[d]^{\Sigma(\theta _{X})}  \\
	T(X) \ar[r]^{T(f)}\ar[d]^{h_{X}} & T(B) \ar[d]^{h_{B}} \ar[r]^{T(g)} & T(Y) \ar[d]^{h_{Y}} \ar[r]^{T(h)} & \Sigma T(X) \ar[d]^{\Sigma(h_{X})}  \\
	X \ar[r]^{f} & B \ar[r]^{g} & Y \ar[r]^{h} & \Sigma X.
}
\]
Again, by \cite[Proposition~1.4.5]{derived}, $h_{X}\circ \theta _{X}$ and $h_{Y}\circ \theta _{Y}$ are the unique morphisms such that the left and right outer squares are commutative, respectively. Since $h_{B}\circ \theta_{B}=\mathrm{Id}_{B}$, we must have $h_{X}\circ \theta_{X}=\mathrm{Id}_{X}$ and $h_{Y}\circ \theta_{Y}=\mathrm{Id}_{Y}$.

We have thus shown that $(X,h_{X})$ and $(Y,h_{Y})$ are both $\mathbf{T}$-modules. Consequently, there exists a distinguished triangle
\[
(X,h_{X})\overset{f}\longrightarrow (B,h_{B}) \overset{g}{\longrightarrow}(Y,h_{Y})\overset{h}{\longrightarrow} \Sigma (X,h_{X}),
\]
in $\mathbf{T}\text{-}\mathrm{Mod}_{\mathcal{B}}$, where $(X,h_{X})\in \mathbf{T}\text{-}\mathrm{Mod}_{\mathcal{B}^{\le 0}}$ and $(Y,h_{Y})\in \Sigma^{-1} (\mathcal{B}^{\ge 0})^{\mathbf{T}}$. Hence, $\left (\mathbf{T}\text{-}\mathrm{Mod}_{\mathcal{B}^{\le 0}},(\mathcal{B}^{\ge 0})^{\mathbf{T}}\right )$ forms a $t$-structure on $\mathbf{T}\text{-}\mathrm{Mod}_{\mathcal{B}}$.
\end{proof}

\begin{rem}\rm
\begin{enumerate}
    \item[(1)] 	In the setting of Theorem~\ref{tot}, the forgetful functor $G^{\mathbf{T}}$ is automatically $t$-exact, and its left adjoint $F^{\mathbf{T}}$ is right $t$-exact. Furthermore, $F^{\mathbf{T}}$ is $t$-exact if and only if $T$ is $t$-exact.
    
    If $T$ is $t$-exact, then $\mathbf{T}$ restricts
    to a monad on $\mathcal{H}_{\mathcal{B}}$, and
    $\mathcal{H}_{\mathbf{T}\text{-}\mathrm{Mod}_{\mathcal{B}}}
    =
    (\mathbf{T}|_{\mathcal{H}_{\mathcal{B}}})
    \text{-}\mathrm{Mod}_{\mathcal{H}_{\mathcal{B}}}$.
     
    \item[(2)] If $(\mathcal{B}^{\le 0},\mathcal{B}^{\ge 0})$ is a nondegenerate (resp., bounded, stable) $t$-structure on $\mathcal{B}$, then by Lemma~\ref{stab} and Theorem~\ref{tot}, so is $\left (\mathbf{T}\text{-}\mathrm{Mod}_{\mathcal{B}^{\le 0}},(\mathcal{B}^{\ge 0})^{\mathbf{T}}\right )$.
\end{enumerate}
\end{rem}

In the remainder of this section, we characterize the $t$-structure on $\mathbf{T}\text{-}\mathrm{Mod}_{\mathcal{B}}$ induced by a compact generator $C$ in $\mathcal{B}$. The following foundational result was proven in \cite[Theorem 4.2]{combal}.

\begin{lem}\label{compact}\rm
Assume that $\mathcal{B}$ satisfies Hypothesis~\ref{ph}, and suppose that $\mathcal{B}$
has a compact generator $C$.
Then $\mathbf{T}\text{-}\mathrm{Mod}_{\mathcal{B}}$ has
a compact generator $F^{\mathbf{T}}(C)$.
If, moreover, $T$ preserves compactness, then
\[
(\mathbf{T}\text{-}\mathrm{Mod}_{\mathcal{B}})^c
=
\mathbf{T}\text{-}\mathrm{Mod}_{\mathcal{B}^c},
\]
and this category is a separable extension of
$\mathcal{B}^c$.
\end{lem}

The following proposition demonstrates that the $t$-structure on $\mathbf{T}\text{-}\mathrm{Mod}_{\mathcal{B}}$ generated by the compact generator $F^{\mathbf{T}}(C)$ is exactly the one induced by the compactly generated $t$-structure associated with $C$ on $\mathcal{B}$.

\begin{prop}\label{comt}\rm
Assume that $\mathcal{B}$ satisfies Hypothesis~\ref{ph}, and suppose that $\mathcal{B}$
has a compact generator $C$. Then $ (\mathbf{T}\text{-}\mathrm{Mod}_{\mathcal{B}})^{\ge 0}_{F^{\mathbf{T}}(C)}=(\mathcal{B}^{\ge 0}_{C})^{\mathbf{T}} $.

 If, moreover, $T$ is right $t$-exact with respect to $(\mathcal{B}^{\le 0}_{C},\mathcal{B}^{\ge 0}_{C})$ (the $t$-structure generated by $C$), then
\[
\left((\mathbf{T}\text{-}\mathrm{Mod}_{\mathcal{B}})^{\le 0}_{F^{\mathbf{T}}(C)},(\mathbf{T}\text{-}\mathrm{Mod}_{\mathcal{B}})^{\ge 0}_{F^{\mathbf{T}}(C)}\right) = \left (\mathbf{T}\text{-}\mathrm{Mod}_{\mathcal{B}^{\le 0}_{C}},(\mathcal{B}^{\ge 0}_{C})^{\mathbf{T}}\right ).
\]
\end{prop}

\begin{proof}
By Lemma~\ref{compact}, $F^{\mathbf{T}}(C)$ is a compact generator of $\mathbf{T}\text{-}\mathrm{Mod}_{\mathcal{B}}$. 
For every $(Y,h_{Y})\in \mathbf{T}\text{-}\mathrm{Mod}_{\mathcal{B}}$ and integer $i>0$, the adjunction gives a natural isomorphism
\[
\mathrm{Hom}_{\mathbf{T}\text{-}\mathrm{Mod}_{\mathcal{B}}}
\bigl(\Sigma^i F^{\mathbf{T}}(C),(Y,h_Y)\bigr)
\cong
\mathrm{Hom}_{\mathcal{B}}(\Sigma^i C,Y).
\]
By the construction of compactly generated $t$-structures, this implies that $(Y,h_{Y})\in (\mathbf{T}\text{-}\mathrm{Mod}_{\mathcal{B}})^{\ge 0}_{F^{\mathbf{T}}(C)}$ if and only if $Y\in \mathcal{B}^{\ge 0}_{C}$, which is precisely the condition that $(Y,h_{Y})\in (\mathcal{B}^{\ge 0}_{C})^{\mathbf{T}}$. This proves the first assertion.

Since $T$ is right $t$-exact, Theorem~\ref{tot} guarantees that $\left (\mathbf{T}\text{-}\mathrm{Mod}_{\mathcal{B}^{\le 0}_{C}},(\mathcal{B}^{\ge 0}_{C})^{\mathbf{T}}\right )$ is a $t$-structure in $\mathbf{T}\text{-}\mathrm{Mod}_{\mathcal{B}}$. Since $ (\mathbf{T}\text{-}\mathrm{Mod}_{\mathcal{B}})^{\ge 0}_{F^{\mathbf{T}}(C)}=(\mathcal{B}^{\ge 0}_{C})^{\mathbf{T}} $, the two $t$-structures coincide.
\end{proof}

\begin{rem}\label{rem310}\rm
Let $(\mathcal{B}^{\le 0},\mathcal{B}^{\ge 0})$ be a $t$-structure on $\mathcal{B}$ in the preferred equivalence class. Then there exists a natural number $n$ such that $\mathcal{B}^{\ge n}_{C}\subseteq \mathcal{B}^{\ge 0}\subseteq \mathcal{B}^{\ge -n}_{C}$. If $T$ is right $t$-exact with respect to $(\mathcal{B}^{\le 0},\mathcal{B}^{\ge 0})$, by Theorem~\ref{tot}, $\left (\mathbf{T}\text{-}\mathrm{Mod}_{\mathcal{B}^{\le 0}},(\mathcal{B}^{\ge 0})^{\mathbf{T}}\right )$ is a $t$-structure on $\mathbf{T}\text{-}\mathrm{Mod}_{\mathcal{B}}$. Moreover, Proposition~\ref{comt} and its shifts give
\[
(\mathbf{T}\text{-}\mathrm{Mod}_{\mathcal{B}})^{\ge n}_{F^{\mathbf{T}}(C)} = (\mathcal{B}^{\ge n}_{C})^{\mathbf{T}} \subseteq (\mathcal{B}^{\ge 0})^{\mathbf{T}} \subseteq (\mathcal{B}^{\ge -n}_{C})^{\mathbf{T}} = (\mathbf{T}\text{-}\mathrm{Mod}_{\mathcal{B}})^{\ge -n}_{F^{\mathbf{T}}(C)}.
\]
Hence the lifted $t$-structure $\left (\mathbf{T}\text{-}\mathrm{Mod}_{\mathcal{B}^{\le 0}},(\mathcal{B}^{\ge 0})^{\mathbf{T}}\right )$ is also a $t$-structure in the preferred equivalence class of $\mathbf{T}\text{-}\mathrm{Mod}_{\mathcal{B}}$.
\end{rem}

\begin{prop}\label{frobenius_exact}\rm
Let $\mathcal{B}$ be a compactly generated triangulated category with a compact generator $C$, and let $\mathbf{T}$ be a smashing exact monad. Suppose that the following two conditions hold:
\begin{enumerate}
    \item[(1)] $T$ preserves compactness;
    \item[(2)] There exists an integer $m \ge 1$ such that $T^2 \cong T^{\oplus m}$.
\end{enumerate}
Then there exists a $t$-structure $(\mathcal{B}^{\le 0}, \mathcal{B}^{\ge 0})$ in the preferred equivalence class on $\mathcal{B}$ such that $T$ is right $t$-exact with respect to $(\mathcal{B}^{\le 0}, \mathcal{B}^{\ge 0})$.
\end{prop}

\begin{proof}
Let $(\mathcal{B}_{C}^{\le 0}, \mathcal{B}_{C}^{\ge 0})$ be the $t$-structure generated by the compact generator $C$. We define a new object in $\mathcal{B}$ by
\[
\widetilde{C} := C \oplus T(C).
\]
Since $C$ is compact and $T$ preserves compactness by assumption (1), $T(C)$ is also compact. It is clear that $\widetilde{C}$ is also a compact generator of $\mathcal{B}$. Let $(\mathcal{B}_{\widetilde{C}}^{\le 0}, \mathcal{B}_{\widetilde{C}}^{\ge 0})$ be the $t$-structure generated by $\widetilde{C}$.
By \cite[Remark~1.19]{compactge}, $(\mathcal{B}_{\widetilde{C}}^{\le 0}, \mathcal{B}_{\widetilde{C}}^{\ge 0})$ belongs to the preferred equivalence class.

Next, we show that $T$ is right $t$-exact with respect to $(\mathcal{B}_{\widetilde{C}}^{\le 0}, \mathcal{B}_{\widetilde{C}}^{\ge 0})$. Applying the functor $T$ to the generator $\widetilde{C}$, we have
$T(\widetilde{C}) = T(C \oplus T(C)) = T(C) \oplus T^2(C)$.
By assumption (2), $T^2 \cong T^{\oplus m}$. Therefore,
$T(\widetilde{C}) \cong T(C) \oplus T(C)^{\oplus m} = T(C)^{\oplus (m+1)}$.
Notice that $T(C)$ is a direct summand of $\widetilde{C} = C \oplus T(C)$. Hence, $T(\widetilde{C})$ is a direct summand of a finite direct sum of copies of $\widetilde{C}$. In particular, we have
$T(\widetilde{C}) \in \mathcal{B}_{\widetilde{C}}^{\le 0}$.
Because $T$ is exact and preserves coproducts, we have
\[
T\left(\mathcal{B}_{\widetilde{C}}^{\le 0}\right)=T\left(\overline{\langle \widetilde{C} \rangle}^{(-\infty, 0]}\right) \subseteq \overline{\langle T(\widetilde{C}) \rangle}^{(-\infty, 0]} \subseteq \mathcal{B}_{\widetilde{C}}^{\le 0},
\]
which proves that $T$ is right $t$-exact with respect to $(\mathcal{B}_{\widetilde{C}}^{\le 0}, \mathcal{B}_{\widetilde{C}}^{\ge 0})$.
\end{proof}

Let $\mathbf{T}=(T,\theta,\mu)$ be a monad on an additive
category $\mathcal{B}$, and let $m\ge 1$ be an integer.
We say that $\mathbf{T}$ is \emph{a monad of rank $m$}
if there is a natural isomorphism
 $T^2 \cong T^{\oplus m}$ of endofunctors of $\mathcal{B}$.
A functor $G:\mathcal{A} \to \mathcal{B}$ is called \emph{Frobenius} if there exists a functor $F:\mathcal{B}\to \mathcal{A}$ which is simultaneously a left and a right adjoint to $G$. A monad $\mathbf{T}$ on $\mathcal{B}$ is called a
\emph{Frobenius monad} if its forgetful functor
$G^{\mathbf{T}}$ is Frobenius.

\begin{cor}\label{fro_con}\rm
Assume that $\mathcal{B}$ satisfies Hypothesis~\ref{ph}, and suppose that
$\mathcal{B}$ has a compact generator $C$. If $\mathbf{T}$ is a Frobenius monad of rank $m$,
then there exists a $t$-structure $(\mathcal{B}^{\le 0}, \mathcal{B}^{\ge 0})$ in the preferred equivalence class on $\mathcal{B}$ such that $T$ is right $t$-exact with respect to $(\mathcal{B}^{\le 0}, \mathcal{B}^{\ge 0})$.
\end{cor}

\begin{proof}
Since $\mathbf{T}$ is an exact Frobenius monad, the forgetful functor $G^{\mathbf{T}}$ is an exact Frobenius functor. This implies that $F^{\mathbf{T}}$ is simultaneously a left and a right adjoint to $G^{\mathbf{T}}$. 

We know from Lemma~\ref{compact} that $\mathbf{T}\text{-}\mathrm{Mod}_{\mathcal{B}}$ is a compactly generated triangulated category, and from Lemma~\ref{coproducts} that $G^{\mathbf{T}}$ preserves coproducts. Under these conditions, \cite[Lemma~5.4.1]{localizationT} ensures that the functor $T=G^{\mathbf{T}}F^{\mathbf{T}}$ preserves compactness. 

Therefore, all assumptions of Proposition~\ref{frobenius_exact} are satisfied. Hence, there exists a $t$-structure $(\mathcal{B}^{\le 0}, \mathcal{B}^{\ge 0})$ in the preferred equivalence class on $\mathcal{B}$ such that $T$ is right $t$-exact with respect to $(\mathcal{B}^{\le 0},\mathcal{B}^{\ge 0})$.
\end{proof}

\section{Preservation of approximability and applications}\label{Section 4}

In this section, we prove that separable extensions preserve
approximability and weak approximability when the associated
monad is smashing and right $t$-exact with respect to a
$t$-structure in the preferred equivalence class.
We then apply this result to smashing localizations and
establish a converse-type result for the gluing of
approximable triangulated categories: under a compactness
assumption, approximability of the middle category in a
recollement implies approximability of the outer categories.
Finally, we apply this result to upper triangular matrix
dg-algebras.

We begin with the main result of this section.

\begin{thm}\label{appro}\rm
Suppose that $\mathcal{B}$ is an approximable \textup{(respectively, weakly approximable)} triangulated category satisfying Hypothesis~\ref{ph}. Let $(\mathcal{B}^{\le 0},\mathcal{B}^{\ge 0})$ be a $t$-structure in the preferred equivalence class. If $T$ is right $t$-exact with respect to $(\mathcal{B}^{\le 0},\mathcal{B}^{\ge 0})$, then $\mathbf{T}\text{-}\mathrm{Mod}_{\mathcal{B}}$ is an approximable \textup{(respectively, weakly approximable)} triangulated category.
\end{thm}

\begin{proof}
Let $C$ be a compact generator for $\mathcal{B}$. According to Lemma~\ref{compact}, $\mathbf{T}\text{-}\mathrm{Mod}_{\mathcal{B}}$ is a compactly generated triangulated category with a compact generator $F^{\mathbf{T}}(C)$. By Remark~\ref{rem310}, $\left (\mathbf{T}\text{-}\mathrm{Mod}_{\mathcal{B}^{\le 0}},(\mathcal{B}^{\ge 0})^{\mathbf{T}}\right)$ is a $t$-structure on $\mathbf{T}\text{-}\mathrm{Mod}_{\mathcal{B}}$ belonging to the preferred equivalence class.

Since $\mathcal{B}$ is approximable, we may choose
an integer $n>0$ such that conditions \textup{(1)}
and \textup{(3)} of Definition~\ref{approximable}
hold for $C$ and $(\mathcal{B}^{\le 0},\mathcal{B}^{\ge 0})$.
In particular,
$\Sigma^n C\in\mathcal{B}^{\le 0}$ and
$\mathrm{Hom}_{\mathcal{B}}
(\Sigma^{-n}C,\mathcal{B}^{\le 0})=0$.
Because $T$ is right $t$-exact, we have $\Sigma^{n} F^{\mathbf{T}}(C) \in\mathbf{T}\text{-}\mathrm{Mod}_{\mathcal{B}^{\le 0}}$. Additionally, we have
\[
\mathrm{Hom}_{\mathbf{T}\text{-}\mathrm{Mod}_{\mathcal{B}}}(\Sigma^{-n} F^{\mathbf{T}}(C),\mathbf{T}\text{-}\mathrm{Mod}_{\mathcal{B}^{\le 0}})\cong \mathrm{Hom}_{\mathcal{B}}\left (\Sigma^{-n} C,G^{\mathbf{T}}\left (\mathbf{T}\text{-}\mathrm{Mod}_{\mathcal{B}^{\le 0}}\right )\right )=0.
\]

We now verify condition (3) in Definition~\ref{approximable}. For any $(B,h_{B}) \in \mathbf{T}\text{-}\mathrm{Mod}_{\mathcal{B}^{\le 0}}$, there exists a distinguished triangle
\[
X\longrightarrow B\longrightarrow Y\longrightarrow \Sigma X
\]
in $\mathcal{B}$, with $X\in \overline{\langle C \rangle}_{n}^{[-n,n]}$ and $Y\in \mathcal{B}^{\le -1}$. By the octahedral axiom, we obtain the following commutative diagram:
\[
\xymatrix{
    & \Sigma^{-1} (B',h_{B'}) \ar[d] \ar@{=}[r] & \Sigma^{-1} (B',h_{B'}) \ar[d] &\\
    F^{\mathbf{T}}(X) \ar[r] \ar@{=}[d] & F^{\mathbf{T}}(B) \ar[d]^{\varepsilon^{\mathbf{T}}_{(B,h_{B})}} \ar[r] & F^{\mathbf{T}}(Y) \ar[d]^{f} \ar[r] & \Sigma F^{\mathbf{T}}(X) \ar@{=}[d] \\
    F^{\mathbf{T}}(X) \ar[r] & (B,h_{B}) \ar[r] \ar[d] & (Y',h_{Y'}) \ar[r] \ar[d]^{g} & \Sigma F^{\mathbf{T}}(X) \\
    & (B',h_{B'}) \ar@{=}[r] & (B',h_{B'}) & 	
}
\]
Note that $\varepsilon^{\mathbf{T}}_{(B,h_{B})}=h_{B}$ is a retraction in $\mathbf{T}\text{-}\mathrm{Mod}_{\mathcal{B}}$ because $G^{\mathbf{T}}$ is separable. Thus, we have $(B',h_{B'})\in \mathbf{T}\text{-}\mathrm{Mod}_{\mathcal{B}^{\le -1}}$. It is easy to see that $F^{\mathbf{T}}(Y)\in \mathbf{T}\text{-}\mathrm{Mod}_{\mathcal{B}^{\le -1}}$, which consequently forces $(Y',h_{Y'})\in \mathbf{T}\text{-}\mathrm{Mod}_{\mathcal{B}^{\le -1}}$. Since  $F^{\mathbf{T}}$ is exact and preserves coproducts, $F^{\mathbf{T}}(X) \in \overline{\langle F^{\mathbf{T}}(C) \rangle}_{n}^{[-n,n]}$. 

Hence, for any $(B,h_{B}) \in \mathbf{T}\text{-}\mathrm{Mod}_{\mathcal{B}^{\le 0}}$, there exists a distinguished triangle
\[
F^{\mathbf{T}}(X)\longrightarrow (B,h_{B})\longrightarrow (Y',h_{Y'})\longrightarrow \Sigma F^{\mathbf{T}}(X)
\]
in $\mathbf{T}\text{-}\mathrm{Mod}_{\mathcal{B}}$, with $F^{\mathbf{T}}(X)\in \overline{\langle F^{\mathbf{T}}(C) \rangle}_{n}^{[-n,n]}$ and $(Y',h_{Y'})\in \mathbf{T}\text{-}\mathrm{Mod}_{\mathcal{B}^{\le -1}}$. 

The proof for the weakly approximable case is analogous and implicit in the above argument.
\end{proof}

\begin{rem}\label{THproof}\rm
We can now complete the proof of Theorem~\ref{thm:MainA} stated in the Introduction. Let $\mathcal{A}$ be an idempotent complete separable extension of $\mathcal{B}$, and let $\mathbf{T}$ be the induced monad on $\mathcal{B}$. As pointed out in Remark~\ref{remsep}(3), the idempotent completeness of $\mathcal{A}$ guarantees that $G$ is monadic; that is, the Eilenberg--Moore comparison functor $\varphi: \mathcal{A} \xrightarrow{\sim} \mathbf{T}\text{-}\mathrm{Mod}_{\mathcal{B}}$ is an equivalence of triangulated categories. Consequently, Theorem~\ref{appro} ensures that $\mathcal{A}$ is an approximable (or weakly approximable) triangulated category.
\end{rem}

\begin{prop}\label{prop_fro}\rm 
  Suppose $\mathcal{B}$ is an approximable \textup{(respectively, weakly approximable)} triangulated category satisfying Hypothesis~\ref{ph}. If $\mathbf{T}$ is a Frobenius monad of rank $m$, then $\mathbf{T}\text{-}\mathrm{Mod}_{\mathcal{B}}$ is an approximable \textup{(respectively, weakly approximable)} triangulated category.
\end{prop}
\begin{proof}
  By Corollary~\ref{fro_con}, there exists a $t$-structure $(\mathcal{B}^{\le 0}, \mathcal{B}^{\ge 0})$ in the preferred equivalence class of $\mathcal{B}$ such that $T$ is right $t$-exact with respect to $(\mathcal{B}^{\le 0}, \mathcal{B}^{\ge 0})$. The conclusion follows from Theorem~\ref{appro}.
\end{proof}

Let $\mathcal{B}$ be a triangulated category and $(T, \lambda)$ a \emph{Bousfield localization}, i.e., $T: \mathcal{B} \to \mathcal{B}$ is an exact functor and $\lambda: \mathrm{Id}_{\mathcal{B}} \to T$ is a natural transformation such that $T\lambda: T \to T^2$ is an isomorphism and $T\lambda = \lambda T$. 
By \cite[Example~6.3]{septri}, the associated monad $\mathbf{T} = (T, \lambda, (T\lambda)^{-1})$ is stably separable, and the category $\mathbf{T}\text{-}\mathrm{Mod}_{\mathcal{B}}$ is canonically equivalent to $\mathrm{Ker}(T)^\perp = \mathrm{Im}(T)$. 
The forgetful functor $G^{\mathbf{T}}: \mathrm{Ker}(T)^\perp \to \mathcal{B}$ is the fully faithful inclusion, with its left adjoint given by $F^{\mathbf{T}} = T: \mathcal{B} \to \mathrm{Ker}(T)^\perp$. 
Since $T^2 \cong T$, it follows that $\mathbf{T}$ is a monad of rank 1. 
Finally, we call $T$ \emph{smashing} if it preserves coproducts, which is equivalent to the monad $\mathbf{T}$ itself being smashing. The following result is an immediate consequence of Theorem~\ref{appro}.

\begin{cor}\label{loc}\rm
Let $\mathcal{B}$ be an approximable \textup{(respectively, weakly approximable)} triangulated category and let
$T:\mathcal{B}\to\mathcal{B}$ be a smashing Bousfield localization functor. Suppose that $T$ is right $t$-exact with respect to a $t$-structure in the preferred equivalence class. Then $\mathrm{Im}(T)$ is an approximable \textup{(respectively, weakly approximable)} triangulated subcategory of $\mathcal{B}$.
\end{cor}

\begin{rem}\rm 
  Let $\mathcal{B}$ and $T$ be as in
  Corollary~\ref{loc}, but without assuming right
  $t$-exactness.
  If $T$ preserves compactness,
  then $T^2\cong T$ allows us to apply
  Proposition~\ref{frobenius_exact} with $m=1$. Hence there exists a $t$-structure $(\mathcal{B}^{\le 0}, \mathcal{B}^{\ge 0})$ in the preferred equivalence class on $\mathcal{B}$ such that $T$ is right $t$-exact with respect to $(\mathcal{B}^{\le 0}, \mathcal{B}^{\ge 0})$.  Corollary~\ref{loc} now ensures that $\mathrm{Im}(T)$ is an approximable \textup{(respectively, weakly approximable)} triangulated subcategory of $\mathcal{B}$.
\end{rem}

Recall that a recollement \cite{BBD} of triangulated categories gives rise to Bousfield localizations. Let $\mathcal{B}, \mathcal{B}_{l}$, and $\mathcal{B}_{r}$ be triangulated categories. Assume that $\mathcal{B}$ admits a recollement relative to $\mathcal{B}_{l}$ and $\mathcal{B}_{r}$ of the following form:
\[
\xymatrix{
	\mathcal{B}_{l} \ar[rr]^{I}  && \mathcal{B} \ar[rr]^{Q} \ar@/^1.1pc/[ll]_{I_{r}}\ar@/_1.1pc/[ll]_{I_{l}}  && \mathcal{B}_{r} \ar@/_1.1pc/[ll]_{Q_{l}} \ar@/^1.1pc/[ll]_{Q_{r}} 
} \qquad (\mathbf{R})
\]
Then $II_l$ and $Q_rQ$, equipped with their adjunction
units, are Bousfield localization functors. If $\mathcal{B}$ admits coproducts, then $II_{l}$ is automatically smashing. We note that $Q_{r}Q$ is smashing if and only if $Q_{r}$ preserves coproducts \cite[Proposition 5.5.1]{localizationT}.

Given $t$-structures $(\mathcal{B}_{l}^{\le 0},\mathcal{B}_{l}^{\ge 0})$ and $(\mathcal{B}_{r}^{\le 0},\mathcal{B}_{r}^{\ge 0})$ on $\mathcal{B}_{l}$ and $\mathcal{B}_{r}$, respectively, the \emph{glued $t$-structure} on $\mathcal{B}$ is $(\mathcal{B}_{gl}^{\le 0},\mathcal{B}_{gl}^{\ge 0})$, where
\[
\mathcal{B}_{gl}^{\le 0}:=\left \{ X\in \mathcal{B} \mid I_{l}(X)\in \mathcal{B}^{\le 0}_{l} \quad \text{and} \quad Q(X)\in \mathcal{B}^{\le 0}_{r} \right \} 
\]
and
\[
\mathcal{B}_{gl}^{\ge 0}:=\left \{ X\in \mathcal{B} \mid I_{r}(X)\in \mathcal{B}^{\ge 0}_{l} \quad \text{and} \quad Q(X)\in \mathcal{B}^{\ge 0}_{r} \right \}.
\]
We point out that the functors $I_{l}$ and $Q_{l}$ are right $t$-exact, $I$ and $Q$ are $t$-exact, and $I_{r}$ and $Q_{r}$ are left $t$-exact, see \cite{BBD}.

By \cite[Theorem 5.1]{glue}, if the outer categories
$\mathcal{B}_l$ and $\mathcal{B}_r$ are approximable
and $\mathcal{B}$ is compactly generated, then
$\mathcal{B}$ has a compact generator $C$.
If, moreover,
\[
\operatorname{Hom}_{\mathcal{B}}
(\Sigma^{-n}C,C)=0
\qquad\text{for }n\gg 0,
\]
then $\mathcal{B}$ is approximable. In the remainder of this section, we investigate a converse-type result to their theorem. 

The proof of the following lemma is adapted from that of
\cite[Lemma~3.11]{glue}. We include it for completeness.

\begin{lem}\label{rem_lem2}\rm
Suppose that we have a recollement as in $(\mathbf{R})$. Let $C$ be a compact generator of $\mathcal{B}$ and $C_{r}$ a compact object of $\mathcal{B}_{r}$. If $Q(C)\in \overline{\langle C _r \rangle}^{(-\infty,0]}$, then we can find a compact object $\widetilde{C}$ in $\mathcal{B}$ such that the $t$-structure generated by $\widetilde{C}$ is equal to the gluing of the $t$-structure on $\mathcal{B}_{l}$ generated by the compact object $I_{l}(\widetilde{C})$ and the $t$-structure on $\mathcal{B}_{r}$ generated by $C_r$.
\end{lem}

\begin{proof}
We note that $Q_{l}(C_{r})$ is compact in $\mathcal{B}$ because the functor $Q_{l}$ preserves compactness. Let 
\[
\widetilde{C}:=C \oplus Q_{l}(C_{r}),
\]
which is clearly a compact object in $\mathcal{B}$. We write $(\mathcal{B}_{l}^{\le 0},\mathcal{B}_{l}^{\ge 0})$ for the $t$-structure on $\mathcal{B}_{l}$ generated by $I_{l}(\widetilde{C})$, $(\mathcal{B}_{r}^{\le 0},\mathcal{B}_{r}^{\ge 0})$ for the $t$-structure on $\mathcal{B}_{r}$ generated by $C_r$, and $(\mathcal{B}_{gl}^{\le 0},\mathcal{B}_{gl}^{\ge 0})$ for the $t$-structure on $\mathcal{B}$ obtained by gluing $(\mathcal{B}_{l}^{\le 0},\mathcal{B}_{l}^{\ge 0})$ and $(\mathcal{B}_{r}^{\le 0},\mathcal{B}_{r}^{\ge 0})$. We claim that $(\mathcal{B}_{gl}^{\le 0},\mathcal{B}_{gl}^{\ge 0})$ is equal to the $t$-structure $(\mathcal{B}_{\widetilde{C}}^{\le 0},\mathcal{B}_{\widetilde{C}}^{\ge 0})$ generated by $\widetilde{C}$. To establish this, we only need to show
\[
\mathcal{B}_{\widetilde{C}}^{\le 0}=\mathcal{B}_{gl}^{\le 0}.
\]
We note that $I_{l}(\widetilde{C})\in \mathcal{B}_{l}^{\le 0}$ and $Q(\widetilde{C})=Q(C)\oplus C_r\in \mathcal{B}_{r}^{\le 0}$. Using the fact that the functors $I_{l}$ and $Q$ preserve coproducts, we deduce that $I_{l}(\mathcal{B}_{\widetilde{C}}^{\le 0})\subseteq \mathcal{B}_{l}^{\le 0}$ and $Q(\mathcal{B}_{\widetilde{C}}^{\le 0}) \subseteq \mathcal{B}_{r}^{\le 0}$. Thus, $\mathcal{B}_{\widetilde{C}}^{\le 0}\subseteq \mathcal{B}_{gl}^{\le 0}$.

On the other hand, we have $Q_{l}(C_r)\in \mathcal{B}_{\widetilde{C}}^{\le 0}$, and thus $Q_{l}(\mathcal{B}_{r}^{\le 0})\subseteq \mathcal{B}_{\widetilde{C}}^{\le 0}$. Since $II_{l}(\widetilde{C})$ fits into a distinguished triangle
\[
Q_{l}Q(\widetilde{C})\longrightarrow \widetilde{C}\longrightarrow II_{l}(\widetilde{C})\longrightarrow \Sigma Q_{l}Q(\widetilde{C}),
\]
with $Q_{l}Q(\widetilde{C})=Q_{l}Q(C)\oplus Q_{l}(C_r)$ and $\widetilde{C}$ both belonging to $\mathcal{B}_{\widetilde{C}}^{\le 0}$, it follows that $II_{l}(\widetilde{C}) \in \mathcal{B}_{\widetilde{C}}^{\le 0}$, and thus $I(\mathcal{B}_{l}^{\le 0})\subseteq \mathcal{B}_{\widetilde{C}}^{\le 0}$. Now, let $Y\in \mathcal{B}_{gl}^{\le 0}$ and consider the distinguished triangle
\[
Q_{l}Q(Y)\longrightarrow Y\longrightarrow II_{l}(Y)\longrightarrow \Sigma Q_{l}Q(Y).
\]
We have $Q_{l}Q(Y)\in Q_{l}(\mathcal{B}_{r}^{\le 0})\subseteq  \mathcal{B}_{\widetilde{C}}^{\le 0}$, and $II_{l}(Y) \in I(\mathcal{B}_{l}^{\le 0})\subseteq \mathcal{B}_{\widetilde{C}}^{\le 0}$. We conclude that $\mathcal{B}_{\widetilde{C}}^{\le 0}=\mathcal{B}_{gl}^{\le 0}$.
\end{proof}

We now obtain the following result, which serves as a converse-type statement to \cite[Theorem 5.1]{glue}.

\begin{prop}\label{recollement}\rm
Let $\mathcal{B}$ be an approximable \textup{(respectively, weakly approximable)} triangulated category with a compact generator $C$, admitting a recollement as in $(\mathbf{R})$. If $Q(C)$ is compact in $\mathcal{B}_r$, then $\mathcal{B}_{l}$ and $\mathcal{B}_{r}$ are both approximable \textup{(respectively, weakly approximable)} triangulated categories.
\end{prop}

\begin{proof}
We present the proof for approximability; the weakly approximable case is analogous. If $Q(C)$ is compact, it is clear that $Q(C)$ is a compact generator in $\mathcal{B}_{r}$. Let $(\mathcal{B}_{r}^{\le 0},\mathcal{B}_{r}^{\ge 0})$ be the $t$-structure on $\mathcal{B}_{r}$ generated by $Q(C)$. We note that $Q_{l}Q(C)$ is a compact object in $\mathcal{B}$ because $Q_{l}$ preserves compactness. It follows that $\mathrm{Hom}_{\mathcal{B}}(\Sigma^{-m} Q_{l}Q(C),C)=0$ for $m\gg 0$ by \cite[Lemma~3.5(2)]{glue}, since we assume that $\mathcal{B}$ is approximable. Then $\mathrm{Hom}_{\mathcal{B}_{r}}(\Sigma^{-m} Q(C),Q(C))=0$ for $m\gg 0$. By \cite[Lemma~3.5(1)]{glue}, there exists an integer $k>0$ such that $\mathrm{Hom}_{\mathcal{B}_{r}}(\Sigma^{-k} Q(C),\mathcal{B}_{r}^{\le 0})=0$.

Let $\widetilde{C}=C \oplus Q_{l}Q(C)$ as in the proof of Lemma~\ref{rem_lem2}. Because $C$ is a compact generator of $\mathcal{B}$, it easily follows that $\widetilde{C}$ is also a compact generator of $\mathcal{B}$. The $t$-structure $(\mathcal{B}_{\widetilde{C}}^{\le 0},\mathcal{B}_{\widetilde{C}}^{\ge 0})$ generated by $\widetilde{C}$ fulfills the approximability criteria. Let $(\mathcal{B}_{l}^{\le 0},\mathcal{B}_{l}^{\ge 0})$ be the $t$-structure on $\mathcal{B}_{l}$ generated by $I_{l}(\widetilde{C})$. Then, by Lemma~\ref{rem_lem2}, $(\mathcal{B}_{\widetilde{C}}^{\le 0},\mathcal{B}_{\widetilde{C}}^{\ge 0})$ is exactly the gluing of $(\mathcal{B}_{l}^{\le 0},\mathcal{B}_{l}^{\ge 0})$ and $(\mathcal{B}_{r}^{\le 0},\mathcal{B}_{r}^{\ge 0})$. We see that both $Q_l$ and $Q$ are right $t$-exact; that is, $Q_{l}(\mathcal{B}_{r}^{\le 0})\subseteq \mathcal{B}_{\widetilde{C}}^{\le 0}$ and $Q(\mathcal{B}_{\widetilde{C}}^{\le 0})\subseteq \mathcal{B}_{r}^{\le 0}$.

Let $X\in \mathcal{B}_{r}^{\le 0}$. There exists an integer $n>0$ such that $Q_{l}(X)$ fits into a distinguished triangle
\[
M\longrightarrow Q_{l}(X)\longrightarrow N\longrightarrow \Sigma M
\]
in $\mathcal{B}$, where $M\in \overline{\langle C \rangle}_{n}^{[-n,n]}$ and $N\in \mathcal{B}_{\widetilde{C}}^{\le -1}$. Since $X\cong QQ_{l}(X)$, we obtain the distinguished triangle
\[
Q(M)\longrightarrow X\longrightarrow Q(N)\longrightarrow \Sigma Q(M)
\]
in $\mathcal{B}_{r}$, where $Q(M)\in \overline{\langle Q(C) \rangle}_{n}^{[-n,n]}$ and $Q(N)\in \mathcal{B}_{r}^{\le -1}$. Let $a= \max\{n,k\}$. We clearly have $\Sigma^{a} Q(C)\in \mathcal{B}_{r}^{\le 0}$, $\mathrm{Hom}_{\mathcal{B}_{r}}(\Sigma^{-a} Q(C),\mathcal{B}_{r}^{\le 0})=0$, and $Q(M)\in \overline{\langle Q(C) \rangle}_{n}^{[-n,n]}\subseteq \overline{\langle Q(C) \rangle}_{a}^{[-a,a]}$. We conclude that $\mathcal{B}_{r}$ is approximable.

For $\mathcal{B}_{l}$, we note that $I_l$ and $I$ are right $t$-exact, and hence the Bousfield localization $II_{l}$ (which is smashing) is right $t$-exact with respect to the $t$-structure $(\mathcal{B}_{\widetilde{C}}^{\le 0},\mathcal{B}_{\widetilde{C}}^{\ge 0})$. Corollary~\ref{loc} implies that
$\mathrm{Im}(II_l)$ is approximable.
Since
$\mathcal{B}_l\simeq\operatorname{Im}(II_l)$, the category $\mathcal{B}_l$ is approximable.
\end{proof}

\begin{rem}\rm
	Let $\mathcal{B}$ be an approximable
	\textup{(respectively, weakly approximable)}
	triangulated category with a compact generator $C$,
	and let $T:\mathcal{B}\to\mathcal{B}$ be a smashing
	Bousfield localization functor.
	Suppose that the inclusion
	$I:\operatorname{Ker}(T)\hookrightarrow\mathcal{B}$
	admits a left adjoint $I_l$.
	By \cite[Proposition~4.13.1]{localizationT},
	$T$ induces a recollement
	\[
	\xymatrix{
		\mathrm{Ker}(T) \ar[rr]^{I}
		&&
		\mathcal{B}
		\ar[rr]^{Q}
		\ar@/^1.1pc/[ll]_{I_r}
		\ar@/_1.1pc/[ll]_{I_l}
		&&
		\mathrm{Im}(T)
		\ar@/_1.1pc/[ll]_{Q_l}
		\ar@/^1.1pc/[ll]_{Q_r}.
	}
	\]
	
	Since $T$ is smashing, $Q_r$ preserves coproducts
	by \cite[Proposition~5.5.1]{localizationT}.
	The adjunction $(Q,Q_r)$ therefore implies
	that $Q$ preserves compactness.
	In particular, $Q(C)$ is compact in
	$\operatorname{Im}(T)$.
	Proposition~\ref{recollement} now shows that
	$\operatorname{Ker}(T)$ and $\operatorname{Im}(T)$
	are both approximable
	\textup{(respectively, weakly approximable)}.
    Thus, $T$ induces a recollement of approximable \textup{(respectively, weakly approximable)} triangulated categories.
\end{rem}
Combining \cite[Theorem~5.1]{glue} with Proposition~\ref{recollement}, we obtain the following conclusion.

\begin{cor}\label{recollementiff}\rm
Suppose that we have a recollement as in $(\mathbf{R})$. If $\mathcal{B}$ has a compact generator $C$ such that $\mathrm{Hom}_{\mathcal{B}}(\Sigma^{-n} C,C)=0$ for $n\gg 0$, and $Q(C)$ is compact, then $\mathcal{B}$ is approximable if and only if  $\mathcal{B}_{l}$ and $\mathcal{B}_{r}$ are both approximable.
\end{cor}
\begin{proof}
	The forward implication follows from
	Proposition~\ref{recollement}.
	The reverse implication follows from
	\cite[Theorem~5.1]{glue}, using the assumed vanishing
	condition on $C$.
\end{proof}
Let $\Lambda$ be a dg-algebra. If its derived category $\mathbf{D}(\Lambda)$ of left dg $\Lambda$-modules is approximable, then $H^{n}(\Lambda)=0$ for $n\gg 0$ by \cite[Lemma~3.5]{glue}. Conversely, if $H^{n}(\Lambda)=0$ for all $n>0$, then $\mathbf{D}(\Lambda)$ is approximable by \cite[Remark~5.3]{compactge}. These observations motivate the study of additional conditions ensuring approximability for cohomologically
bounded-above dg-algebras. 

We fix a base commutative ring $k$. Let $R$ and $S$ be dg-algebras with ${}_R M_S$ a dg-bimodule. Let
\[
\Lambda=\begin{bmatrix}
  R & M\\
  0 & S
\end{bmatrix}
\]
denote the upper triangular matrix dg-algebra with the differential
\[
d_{\Lambda}\left (\begin{bmatrix}
  r & m\\
  0 & s
\end{bmatrix}\right )=\begin{bmatrix}
  d_{R}(r) & d_{M}(m)\\
  0 & d_{S}(s)
\end{bmatrix}.
\]

Let $e_{R}=\begin{bmatrix} 1 & 0\\ 0 & 0 \end{bmatrix}$ and $e_{S}=\begin{bmatrix} 0 & 0\\ 0 & 1 \end{bmatrix}$, and define the left dg-$\Lambda$-modules
\[
B:=\Lambda e_{R}=\begin{bmatrix} R\\ 0 \end{bmatrix} \quad \text{and} \quad C:=\Lambda e_{S}=\begin{bmatrix} M\\ S \end{bmatrix},
\]
where $B$ has the differential $d\left (\begin{bmatrix} r\\ 0 \end{bmatrix}\right )=\begin{bmatrix} d_{R}(r)\\ 0 \end{bmatrix}$ and $C$ has the differential $d\left (\begin{bmatrix} m\\ s \end{bmatrix}\right )=\begin{bmatrix} d_{M}(m)\\ d_{S}(s) \end{bmatrix}$. 

We conclude with an application to upper triangular matrix dg-algebras, relating the approximability of $\mathbf{D}(\Lambda)$ to that of the derived categories $\mathbf{D}(R)$ and $\mathbf{D}(S)$ of its diagonal
dg-algebras.

\begin{cor}\rm
Let $\Lambda$ be an upper triangular matrix dg-algebra constructed from dg-algebras $R$ and $S$ as above. Suppose that ${}_R M_S$ is quasi-isomorphic to some ${}_R N_S$ where $N$ is $K$-projective as a left dg-$R$-module (this condition is automatically satisfied when $k$ is a field). If $\mathbf{D}(\Lambda)$ is approximable, then $\mathbf{D}(R)$ and $\mathbf{D}(S)$ are both approximable. If, moreover, $H^{n}(M)=0$ for $n\gg 0$, then $\mathbf{D}(\Lambda)$ is approximable if and only if $\mathbf{D}(R)$ and $\mathbf{D}(S)$ are approximable.
\end{cor}

\begin{proof}
By \cite[Section~3]{Dguppertri} (see also \cite[Theorem~3.3]{Recollementofdg}), there is a recollement
\[
\xymatrix{
	\mathbf{D}(R) \ar[rr]^{I}  && \mathbf{D}(\Lambda) \ar[rr]^{Q} \ar@/^1.1pc/[ll]_{I_{r}}\ar@/_1.1pc/[ll]_{I_{l}}  && \mathbf{D}(S) \ar@/_1.1pc/[ll]_{Q_{l}} \ar@/^1.1pc/[ll]_{Q_{r}}
}
\]
where $Q_{l}={}_\Lambda C_{S}\otimes ^{\mathbf{L}}_{S}-$, $Q=\mathbf{R}\mathrm{Hom}_{\Lambda}({}_\Lambda C_{S},-)$, $Q_{r}=\mathbf{R}\mathrm{Hom}_{S}(Q(\Lambda),-)$, $I={}_\Lambda B_{R}\otimes^{\mathbf{L}}_{R}-$, and $I_{r}=\mathbf{R}\mathrm{Hom}_{\Lambda}({}_\Lambda B_{R},-)$. By \cite[Lemma~3.11]{Dguppertri}, we have $Q(\Lambda)\cong {}_{S}S$ in $\mathbf{D}(S)$, which is a compact object. If $\mathbf{D}(\Lambda)$ is approximable, then by Proposition~\ref{recollement}, we deduce that $\mathbf{D}(R)$ and $\mathbf{D}(S)$ are both approximable. The second assertion is a direct consequence of \cite[Corollary 6.1]{glue}.
\end{proof}

\section{The intrinsic subcategories}\label{section5}

Throughout this section, let $\mathcal{B}$ be a weakly
approximable triangulated category satisfying
Hypothesis~\ref{ph}, and fix a $t$-structure
$(\mathcal{B}^{\le 0},\mathcal{B}^{\ge 0})$ in its
preferred equivalence class.
Following \cite[Definition~1.20]{compactge}, define
\[
\mathcal{B}_c^-
:=
\bigcap_{n>0}
\bigl(\mathcal{B}^c*\mathcal{B}^{\le -n}\bigr),
\qquad
\mathcal{B}_c^b
:=
\mathcal{B}_c^-\cap\mathcal{B}^b.
\]
The subcategories
\[
\mathcal{B}^-,
\quad
\mathcal{B}^+,
\quad
\mathcal{B}^b,
\quad
\mathcal{B}_c^-,
\quad
\mathcal{B}_c^b
\]
are thick and independent of the choice of
$t$-structure within the preferred equivalence class;
see \cite{compactge}.

We briefly recall the relation with the intrinsic
subcategories studied in \cite{KPV}.
The subcategories $\mathcal{B}^-$, $\mathcal{B}^+$,
and $\mathcal{B}^b$ agree with those defined there
using vanishing conditions on morphisms from compact
objects; see \cite[Corollary~3.10]{glue} and
\cite[Remark~3.15]{KPV}.
The comparison for $\mathcal{B}_c^b$ requires
additional hypotheses.
Specifically, suppose that $\mathcal{B}$ is
approximable and $k$-linear over a commutative noetherian ring $k$, and that it has a compact generator $C$ such that
$\operatorname{Hom}_{\mathcal{B}}(C,\Sigma^iC)$
is a finitely generated $k$-module for every $i\in\mathbb{Z}$.
Then \cite[Theorem~1.4(i)]{compactge} identifies $\mathcal{B}_c^b$ with the subcategory of bounded finite objects of \cite[Definition~5.1]{KPV}; see also \cite[Remark~5.4]{KPV}. Explicitly, under these additional hypotheses,
\[
\mathcal{B}_c^b
=
\left\{
X\in\mathcal{B}
\ \middle|\
\bigoplus_{i\in\mathbb{Z}}
\mathrm{Hom}_{\mathcal{B}}(P,\Sigma^iX)
\text{ is finitely generated over }k
\text{ for every }P\in\mathcal{B}^c
\right\}.
\]
These additional assumptions are used only for this comparison and are not imposed throughout the section.

We now study restrictions of separable extensions to these subcategories. Their relation to the results of \cite{KK} is discussed after Theorem~\ref{subcat}.
Whenever $T$ is right $t$-exact, we equip
$\mathbf{T}\text{-}\mathrm{Mod}_{\mathcal{B}}$ with the lifted $t$-structure.
By Theorem~\ref{appro} and Remark~\ref{rem310},
$\mathbf{T}\text{-}\mathrm{Mod}_{\mathcal{B}}$ is then weakly approximable and this
$t$-structure belongs to its preferred equivalence class.

\begin{lem}\label{rb}\rm
Suppose that $T$ is right $t$-exact with respect to $(\mathcal{B}^{\le 0},\mathcal{B}^{\ge 0})$. Then $(\mathbf{T}\text{-}\mathrm{Mod}_{\mathcal{B}})^{-}=({\mathcal{B}^{-}})^{\mathbf{T}}$, $(\mathbf{T}\text{-}\mathrm{Mod}_{\mathcal{B}})^{+}=({\mathcal{B}^{+}})^{\mathbf{T}}$, $(\mathbf{T}\text{-}\mathrm{Mod}_{\mathcal{B}})^{b}=({\mathcal{B}^{b}})^{\mathbf{T}}$. Moreover, $(\mathbf{T}\text{-}\mathrm{Mod}_{\mathcal{B}})^{-}=\mathbf{T}\text{-}\mathrm{Mod}_{\mathcal{B}^{-}}$ is a separable extension of $\mathcal{B}^{-}$. If, in addition, $T$ is $t$-exact, then $(\mathbf{T}\text{-}\mathrm{Mod}_{\mathcal{B}})^{+}=\mathbf{T}\text{-}\mathrm{Mod}_{\mathcal{B}^{+}}$ and $(\mathbf{T}\text{-}\mathrm{Mod}_{\mathcal{B}})^{b}=\mathbf{T}\text{-}\mathrm{Mod}_{\mathcal{B}^{b}}$, and these are separable extensions of $\mathcal{B}^{+}$ and $\mathcal{B}^{b}$, respectively.
\end{lem}

\begin{proof}
Taking unions gives
\[
\begin{aligned}
	(\mathbf{T}\text{-}\mathrm{Mod}_{\mathcal{B}})^{-}
	&=
	\bigcup_{n>0}
	(\mathcal{B}^{\le n})^{\mathbf{T}}
	=
	(\mathcal{B}^-)^{\mathbf{T}},\\
	(\mathbf{T}\text{-}\mathrm{Mod}_{\mathcal{B}})^{+}
	&=
	\bigcup_{n>0}
	(\mathcal{B}^{\ge -n})^{\mathbf{T}}
	=
	(\mathcal{B}^+)^{\mathbf{T}}.
\end{aligned}
\]
Consequently, $(\mathbf{T}\text{-}\mathrm{Mod}_{\mathcal{B}})^{b}=(\mathbf{T}\text{-}\mathrm{Mod}_{\mathcal{B}})^{-} \cap  (\mathbf{T}\text{-}\mathrm{Mod}_{\mathcal{B}})^{+}= ({\mathcal{B}^{b}})^{\mathbf{T}}$.

Right $t$-exactness implies
$T(\mathcal{B}^{\le n})\subseteq\mathcal{B}^{\le n}$
for every integer $n$, and hence $	T(\mathcal{B}^-)\subseteq\mathcal{B}^-$.  Lemma~\ref{stab} identifies $(\mathbf{T}\text{-}\mathrm{Mod}_{\mathcal{B}})^{-}$ with
$\mathbf{T}\text{-}\mathrm{Mod}_{\mathcal{B}^-}$ and shows that it is a separable extension of $\mathcal{B}^-$. If $T$ is $t$-exact, it also preserves $\mathcal{B}^+$ and $\mathcal{B}^b$. Applying Lemma~\ref{stab} to these thick subcategories proves the remaining assertions.
\end{proof}

\begin{lem}\label{prop45}\rm
Suppose that $T$ is right $t$-exact with respect to $(\mathcal{B}^{\le 0},\mathcal{B}^{\ge 0})$. Then $(\mathcal{B}^{-}_{c})^{\mathbf{T}}\subseteq (\mathbf{T}\text{-}\mathrm{Mod}_{\mathcal{B}})^{-}_{c}$. If, moreover, $T$ preserves compactness, then $(\mathbf{T}\text{-}\mathrm{Mod}_{\mathcal{B}})^{-}_{c}=\mathbf{T}\text{-}\mathrm{Mod}_{\mathcal{B}^{-}_{c}}$, and this category is a separable extension of
$\mathcal{B}_c^-$.
\end{lem}

\begin{proof}
Let $(X,h_{X})\in (\mathcal{B}^{-}_{c})^{\mathbf{T}}$. For any integer $n>0$, there exists a distinguished triangle
\[
 X^{c}\longrightarrow X\longrightarrow Y\longrightarrow \Sigma X^{c}
 \]
  in $\mathcal{B}$, where $X^{c}\in \mathcal{B}^{c}$ and $Y\in \mathcal{B}^{\le -n}$. The forgetful functor $G^{\mathbf{T}}$ preserves
coproducts, so its left adjoint $F^{\mathbf{T}}$
preserves compactness. Then  $F^{\mathbf{T}}(X^{c})\in (\mathbf{T}\text{-}\mathrm{Mod}_{\mathcal{B}})^{c}$. Since $T$ is right $t$-exact, we have $F^{\mathbf{T}}(Y)\in \mathbf{T}\text{-}\mathrm{Mod}_{\mathcal{B}^{\le -n}}$. This yields $F^{\mathbf{T}}(X)\in (\mathbf{T}\text{-}\mathrm{Mod}_{\mathcal{B}})^{-}_{c}$. 
The separability of $G^{\mathbf{T}}$ implies that the counit
$F^{\mathbf{T}}(X)\longrightarrow(X,h_X)$
is a split epimorphism in $\mathbf{T}\text{-}\mathrm{Mod}_{\mathcal{B}}$. Thus $(X,h_X)$ is a direct summand of $F^{\mathbf{T}}(X)$. Because $(\mathbf{T}\text{-}\mathrm{Mod}_{\mathcal{B}})^{-}_{c}$ is thick, we conclude that $(X,h_{X})\in (\mathbf{T}\text{-}\mathrm{Mod}_{\mathcal{B}})^{-}_{c}$.

If $T$ preserves compactness, we have $T(\mathcal{B}^{-}_{c}) \subseteq \mathcal{B}^{-}_{c}$. By Lemma~\ref{compact}, $\mathbf{T}\text{-}\mathrm{Mod}_{\mathcal{B}^{c}} = (\mathbf{T}\text{-}\mathrm{Mod}_{\mathcal{B}})^{c}$. The equality $(\mathbf{T}\text{-}\mathrm{Mod}_{\mathcal{B}})^{-}_{c} = \mathbf{T}\text{-}\mathrm{Mod}_{\mathcal{B}^{-}_{c}}$ then easily follows from the definition of $(\mathbf{T}\text{-}\mathrm{Mod}_{\mathcal{B}})^{-}_{c}$.
\end{proof}

\begin{lem}\label{prop46}\rm
Suppose that $T$ is right $t$-exact with respect to $(\mathcal{B}^{\le 0},\mathcal{B}^{\ge 0})$. Then $(\mathcal{B}^{b}_{c})^{\mathbf{T}}\subseteq (\mathbf{T}\text{-}\mathrm{Mod}_{\mathcal{B}})^{b}_{c}$. If, moreover, $T$ is $t$-exact and preserves compactness, then $(\mathbf{T}\text{-}\mathrm{Mod}_{\mathcal{B}})^{b}_{c}=\mathbf{T}\text{-}\mathrm{Mod}_{\mathcal{B}^{b}_{c}}$, and this category is a separable extension of
$\mathcal{B}_c^b$.
\end{lem}

\begin{proof}
By Lemmas~\ref{rb} and \ref{prop45}, we have
$(\mathbf{T}\text{-}\mathrm{Mod}_{\mathcal{B}})^b=(\mathcal{B}^b)^{\mathbf{T}}$,
and
$(\mathcal{B}_c^-)^{\mathbf{T}}
\subseteq(\mathbf{T}\text{-}\mathrm{Mod}_{\mathcal{B}})_c^-$.
Consequently,	
	\[
	\begin{aligned}
		(\mathcal{B}_c^b)^{\mathbf{T}}
		=
		(\mathcal{B}_c^-\cap\mathcal{B}^b)^{\mathbf{T}}
		=
		(\mathcal{B}_c^-)^{\mathbf{T}}
		\cap(\mathcal{B}^b)^{\mathbf{T}}
		\subseteq
		(\mathbf{T}\text{-}\mathrm{Mod}_{\mathcal{B}})_c^-\cap(\mathbf{T}\text{-}\mathrm{Mod}_{\mathcal{B}})^b
		= (\mathbf{T}\text{-}\mathrm{Mod}_{\mathcal{B}})_c^b.
	\end{aligned}
	\]
Now suppose that $T$ is $t$-exact and preserves compactness.
By Lemma~\ref{prop45},
$(\mathbf{T}\text{-}\mathrm{Mod}_{\mathcal{B}})_c^-
=
(\mathcal{B}_c^-)^{\mathbf{T}}$.
Combining this with Lemma~\ref{rb} gives \[
(\mathbf{T}\text{-}\mathrm{Mod}_{\mathcal{B}})_c^b=(\mathbf{T}\text{-}\mathrm{Mod}_{\mathcal{B}})_c^-\cap(\mathbf{T}\text{-}\mathrm{Mod}_{\mathcal{B}})^b=(\mathcal{B}_c^-)^{\mathbf{T}}
\cap(\mathcal{B}^b)^{\mathbf{T}}=(\mathcal{B}^b_c)^{\mathbf{T}}.
\]
Moreover, the proof of Lemma~\ref{prop45} shows that
$T$ preserves $\mathcal{B}_c^-$, while $t$-exactness implies that it preserves $\mathcal{B}^b$. Hence
$T(\mathcal{B}_c^b)\subseteq\mathcal{B}_c^b$.
Lemma~\ref{stab} identifies $(\mathcal{B}_c^b)^{\mathbf{T}}$ with
$\mathbf{T}\text{-}\mathrm{Mod}_{\mathcal{B}_c^b}$ and shows that it is a separable extension of $\mathcal{B}_c^b$.
\end{proof}

We collect the preceding results in the following theorem.
\begin{thm}\label{subcat}\rm
Let $\mathcal{B}$ be a weakly approximable triangulated category satisfying Hypothesis~\ref{ph}, and let $(\mathcal{B}^{\le 0},\mathcal{B}^{\ge 0})$ be a t-structure in the preferred equivalence class. Suppose that $T$ is $t$-exact with respect to $(\mathcal{B}^{\le 0},\mathcal{B}^{\ge 0})$ and preserves compactness. Then we have:
\begin{enumerate}
    \item[(1)] $(\mathbf{T}\text{-}\mathrm{Mod}_{\mathcal{B}})^{-}=\mathbf{T}\text{-}\mathrm{Mod}_{\mathcal{B}^{-}}$;
    \item[(2)] $(\mathbf{T}\text{-}\mathrm{Mod}_{\mathcal{B}})^{+}=\mathbf{T}\text{-}\mathrm{Mod}_{\mathcal{B}^{+}}$;
    \item[(3)] $(\mathbf{T}\text{-}\mathrm{Mod}_{\mathcal{B}})^{b}=\mathbf{T}\text{-}\mathrm{Mod}_{\mathcal{B}^{b}}$;
    \item[(4)] $(\mathbf{T}\text{-}\mathrm{Mod}_{\mathcal{B}})^{c}= \mathbf{T}\text{-}\mathrm{Mod}_{\mathcal{B}^{c}}$;
    \item[(5)] $(\mathbf{T}\text{-}\mathrm{Mod}_{\mathcal{B}})^{-}_{c}= \mathbf{T}\text{-}\mathrm{Mod}_{\mathcal{B}^{-}_{c}}$;
    \item[(6)] $(\mathbf{T}\text{-}\mathrm{Mod}_{\mathcal{B}})^{b}_{c}= \mathbf{T}\text{-}\mathrm{Mod}_{\mathcal{B}^{b}_{c}}$.
\end{enumerate}
Each category on the left-hand side is a separable extension of the corresponding subcategory of $\mathcal{B}$.
\end{thm}
\begin{proof}
	Statements \textup{(1)--(3)} follow from
	Lemma~\ref{rb}, and \textup{(4)} follows from
	Lemma~\ref{compact}.
	Statements \textup{(5)} and \textup{(6)} follow from
	Lemmas~\ref{prop45} and \ref{prop46}, respectively.
\end{proof}

\begin{rem}\rm
	In Theorem~\ref{subcat}, part \textup{(3)} is a special
	case of \cite[Theorem~3.3(i)]{KK}, whose proof also yields parts \textup{(1)} and \textup{(2)} using
	one-sided vanishing. Part \textup{(4)} is due to
	\cite[Theorem~4.2]{combal}.
	Part \textup{(5)} concerns Neeman's subcategory $\mathcal{B}_c^-$, which is not treated in \cite{KK}.
	Part \textup{(6)} also follows from
	\cite[Theorem~3.3(iv)]{KK} under additional hypotheses
	ensuring that $\mathcal{B}_c^b$ is the bounded finite subcategory of $\mathcal{B}$ and that
	$(\mathbf{T}\text{-}\mathrm{Mod}_{\mathcal{B}})_c^b$ is the bounded finite subcategory of
	$\mathbf{T}\text{-}\mathrm{Mod}_{\mathcal{B}}$, in the sense of \cite{KPV}.
\end{rem}

\begin{exa}\rm
	Let $Q$ be a finite group acting strictly on a weakly
	approximable triangulated category $\mathcal{B}$ by
	exact automorphisms $U_g$, compatibly with the shift functor. Assume that $|Q|$ is invertible in $\mathcal{B}$.
	
	A \emph{$Q$-equivariant object} is a pair
	$(X,\{\alpha_g\}_{g\in Q})$, where $X\in\mathcal{B}$
	and the isomorphisms $\alpha_g:X\to U_g(X)$ satisfy
	$\alpha_{gh}=U_g(\alpha_h)\circ\alpha_g$ for all $g,h\in Q$.
	A morphism
	$f:(X,\{\alpha_g\}_{g\in Q})
	\longrightarrow (Y,\{\beta_g\}_{g\in Q})$
	is a morphism $f:X\to Y$ in $\mathcal{B}$ satisfying
	$\beta_g\circ f=U_g(f)\circ\alpha_g$
    for all $g\in Q$.
	The resulting category $\mathcal{B}^Q$ is called the
	\emph{equivariantization} of $\mathcal{B}$.
	
	There are an obvious forgetful functor $G:\mathcal{B}^{Q}\to \mathcal{B}$, forgetting the $Q$-equivariant structure, and an induction functor $F:\mathcal{B} \to \mathcal{B}^{Q}$ such that
	$F(X)=\bigoplus_{g\in Q}U_{g}(X)$ equipped with an obvious equivariant structure. By \cite[Lemma~2.14(1)]{equiva}, $F$ is left and right adjoint to $G$.
	This yields a monad $\mathbf{T}$ defined by the adjoint pair $(F,G)$, where the endofunctor $T=GF=\bigoplus_{g\in Q}U_{g}$. 
	By \cite[Lemmas~4.3 and~4.4]{Anote}, the category
	$\mathcal{B}^Q$ identifies with
	$\mathbf{T}\text{-}\mathrm{Mod}_{\mathcal{B}}$,
	and $\mathbf{T}$ is stably separable.
	Under this identification, $\mathbf{T}$ is a
	Frobenius monad. Moreover,
	\[
	T^2
	\cong \bigoplus_{g,h\in Q}U_{gh}
	\cong T^{\oplus |Q|},
	\]
	so $\mathbf{T}$ has rank $|Q|$.
	
	Each $U_g$ preserves coproducts and compact objects,
	and hence so does $T$. Under the standing
	Hypothesis~\ref{ph}, Proposition~\ref{prop_fro}
	therefore implies that $\mathcal{B}^Q$ is weakly
	approximable.
	
	By Corollary~\ref{fro_con}, we may choose a $t$-structure
	$(\mathcal{B}^{\le 0},\mathcal{B}^{\ge 0})$
	in the preferred equivalence class with respect to which $T$ is right $t$-exact.
	We claim that $T$ is also left $t$-exact.
	Indeed, for $X\in\mathcal{B}^{\le -1}$ and
	$Y\in\mathcal{B}^{\ge 0}$, there are natural
	isomorphisms
	\[
	\begin{aligned}
		\mathrm{Hom}_{\mathcal{B}}(X,T(Y))
		\cong
		\bigoplus_{g\in Q}
		\mathrm{Hom}_{\mathcal{B}}(X,U_g(Y))
		\cong
		\bigoplus_{g\in Q}
		\mathrm{Hom}_{\mathcal{B}}(U_{g^{-1}}(X),Y)
		\cong
		\mathrm{Hom}_{\mathcal{B}}(T(X),Y)
		=0,
	\end{aligned}
	\]
	where the last equality follows from
	$T(X)\in\mathcal{B}^{\le -1}$.
	Thus $T(Y)\in\mathcal{B}^{\ge 0}$, proving the claim.
	
	Consequently, $T$ is $t$-exact and preserves compactness. 
	Applying Theorem~\ref{subcat}, we obtain
	\begin{enumerate}
		\item[(1)]
		$(\mathcal{B}^{Q})^{-}=(\mathcal{B}^{-})^{Q}$;
		\item[(2)]
		$(\mathcal{B}^{Q})^{+}=(\mathcal{B}^{+})^{Q}$;
		\item[(3)]
		$(\mathcal{B}^{Q})^{b}=(\mathcal{B}^{b})^{Q}$;
		\item[(4)]
		$(\mathcal{B}^{Q})^{c}=(\mathcal{B}^{c})^{Q}$;
		\item[(5)]
		$(\mathcal{B}^{Q})_{c}^{-}=(\mathcal{B}_{c}^{-})^{Q}$;
		\item[(6)]
		$(\mathcal{B}^{Q})_{c}^{b}=(\mathcal{B}_{c}^{b})^{Q}$.
	\end{enumerate}
	
	The identifications of compact and bounded objects
	also appear in \cite[Corollary~4.1]{KK}, which further
	treats bounded finite objects. Its latter identification
	agrees with \textup{(6)} when the bounded finite
	subcategories coincide with Neeman's subcategories
	$\mathcal{B}_c^b$ and $(\mathcal{B}^Q)_c^b$, respectively.
\end{exa}

\begin{exa}\rm
	Let $\Lambda$ be an associative ring with identity and
	$Q$ a finite group. Let
	$\sigma:Q\to\mathrm{Aut}(\Lambda)$ and
	$\alpha:Q\times Q\to U(\Lambda)$ be maps satisfying
	the usual compatibility and normalization conditions,
	where $U(\Lambda)$ denotes the group of units of
	$\Lambda$. The associated ring
	$\Lambda^{\sigma}_{\alpha}Q$ is called the
	\emph{crossed product} of $\Lambda$ with $Q$.
	When $\alpha$ or $\sigma$ is trivial, we recover
	the skew group ring $\Lambda^{\sigma}Q$ or the
	twisted group ring $\Lambda_{\alpha}Q$, respectively.
	
	Let $H\le Q$ be a subgroup. Restricting $\sigma$ and
	$\alpha$ defines a crossed product
	$\Lambda^{\sigma}_{\alpha}H$, which is naturally a
	subring of $\Lambda^{\sigma}_{\alpha}Q$.
	The latter is finite free as both a left and a right
	$\Lambda^{\sigma}_{\alpha}H$-module, and induction
	is naturally isomorphic to coinduction.
	Assume that $|Q:H|$ is invertible in $\Lambda$.
	Then $\Lambda^{\sigma}_{\alpha}Q$ is a separable
	extension of $\Lambda^{\sigma}_{\alpha}H$.
	For further details, see \cite{crosse}.
	
	The induction functor
	$F=\Lambda^{\sigma}_{\alpha}Q
	\otimes_{\Lambda^{\sigma}_{\alpha}H}-$
	and the restriction functor $G$ are exact and preserve
	coproducts and finitely generated projective modules.
	Let $\mathbf{T}$ be the monad defined by $(F,G)$.
	It induces a stably separable monad
	$\mathbf{D}(\mathbf{T})$ on
	$\mathbf{D}(\Lambda^{\sigma}_{\alpha}H)$,
	which we also denote by $\mathbf{T}$.
	
	The separability splitting of the counit passes
	to the derived categories. Since
	$\mathbf{D}(\Lambda^{\sigma}_{\alpha}Q)$ is
	idempotent complete, the comparison functor yields
	an equivalence
	\[
	\mathbf{D}(\Lambda^{\sigma}_{\alpha}Q)
	\simeq
	\mathbf{T}\text{-}\mathrm{Mod}_{
		\mathbf{D}(\Lambda^{\sigma}_{\alpha}H)}.
	\]
	Moreover, the underlying functor $T$ preserves
	coproducts and compact objects and is $t$-exact
	with respect to the standard $t$-structure.
	Since $\mathbf{D}(\Lambda^{\sigma}_{\alpha}H)$
	is approximable and its standard $t$-structure
	belongs to the preferred equivalence class,
	Theorem~\ref{subcat} gives the following table.
	Each row identifies the middle category as a
	separable extension of the category on the right.
	
	Here $\mathrm{proj}$ denotes finitely generated
	projective left modules, and $\mathbf{K}^{-,b}$
	denotes the homotopy category of bounded above
	complexes with bounded cohomology. We identify
	these homotopy categories with their essential
	images in the corresponding derived categories.
	
	\begin{center}
		\renewcommand{\arraystretch}{1.5}
		\setlength{\tabcolsep}{12pt}
		\begin{tabular}{|*{3}{c|}}
			\hline
			$\mathcal{B}$
			& $\mathbf{T}\text{-}\mathrm{Mod}_{\mathbf{D}(\Lambda^{\sigma}_{\alpha}H)}
			\cong \mathbf{D}(\Lambda^{\sigma}_{\alpha}Q)$
			& $\mathbf{D}(\Lambda^{\sigma}_{\alpha}H)$ \\
			\hline
			$\mathcal{B}^{-}$
			& $\mathbf{D}^{-}(\Lambda^{\sigma}_{\alpha}Q)$
			& $\mathbf{D}^{-}(\Lambda^{\sigma}_{\alpha}H)$ \\
			
			$\mathcal{B}^{+}$
			& $\mathbf{D}^{+}(\Lambda^{\sigma}_{\alpha}Q)$
			& $\mathbf{D}^{+}(\Lambda^{\sigma}_{\alpha}H)$ \\
			
			$\mathcal{B}^{b}$
			& $\mathbf{D}^{b}(\Lambda^{\sigma}_{\alpha}Q)$
			& $\mathbf{D}^{b}(\Lambda^{\sigma}_{\alpha}H)$ \\
			
			$\mathcal{B}^{c}$
			& $\mathbf{K}^{b}(\Lambda^{\sigma}_{\alpha}Q\text{-}\mathrm{proj})$
			& $\mathbf{K}^{b}(\Lambda^{\sigma}_{\alpha}H\text{-}\mathrm{proj})$ \\
			
			$\mathcal{B}^{-}_{c}$
			& $\mathbf{K}^{-}(\Lambda^{\sigma}_{\alpha}Q\text{-}\mathrm{proj})$
			& $\mathbf{K}^{-}(\Lambda^{\sigma}_{\alpha}H\text{-}\mathrm{proj})$ \\
			
			$\mathcal{B}^{b}_{c}$
			& $\mathbf{K}^{-,b}(\Lambda^{\sigma}_{\alpha}Q\text{-}\mathrm{proj})$
			& $\mathbf{K}^{-,b}(\Lambda^{\sigma}_{\alpha}H\text{-}\mathrm{proj})$ \\
			\hline
		\end{tabular}
	\end{center}
	
	The skew group ring case, obtained when $H=\{1\}$
	and $\alpha$ is trivial, is also considered in
	\cite[Section~4.1.1]{KK} in connection with intrinsic
	subcategories and homological invariants.
\end{exa}

\begin{exa}\rm
	Let $f:X\to Y$ be a finite \'etale morphism of
	quasi-compact and quasi-separated schemes.
	Write $\mathbf{D}_{qc}(X)$ for the derived category
	of complexes of $\mathcal{O}_X$-modules with
	quasi-coherent cohomology, and similarly for $Y$. By \cite[Theorem~3.5]{combal}, the adjunction $(f^*, f_*)$ defines a separable exact monad $\mathbf{T}$ on $\mathbf{D}_{qc}(Y)$, and there is a unique equivalence of triangulated
	categories
	\[
	\mathbf{D}_{qc}(X)
	\simeq
	\mathbf{T}\text{-}\mathrm{Mod}_{\mathbf{D}_{qc}(Y)}.
	\]
	Since $f$ is finite \'etale, $f_*\mathcal{O}_X$ is a flat and locally finite generated $\mathcal{O}_Y$-module,
	and 
	$T\cong
	f_*\mathcal{O}_X\otimes-$.
	Thus $T$ preserves coproducts and perfect complexes
	and is $t$-exact for the standard $t$-structure;
	see also \cite[Corollary~4.8]{combal}.
	
	The category $\mathbf{D}_{qc}(Y)$ is weakly approximable, and its standard $t$-structure belongs to the
	preferred equivalence class.
	Since $f_*$ is $t$-exact and conservative, the lifted $t$-structure corresponds to the
	standard $t$-structure on $\mathbf{D}_{qc}(X)$.
	Moreover, the subcategories
	$\mathbf{D}_{qc}(X)_c^-$ and
	$\mathbf{D}_{qc}(X)_c^b$ identify with the
	pseudo-coherent complexes and those with bounded
	cohomology, respectively, and similarly for $Y$;
	see \cite[Remark~5.8]{compactge}.
	Theorem~\ref{subcat} therefore yields the following
	table.
	
	\begin{center}
		\renewcommand{\arraystretch}{1.5}
		\setlength{\tabcolsep}{12pt}
		\begin{tabular}{|*{3}{c|}}
			\hline
			$\mathcal{B}$
			& $\mathbf{D}_{qc}(X)$
			& $\mathbf{D}_{qc}(Y)$ \\
			\hline
			$\mathcal{B}^{-}$
			& $\mathbf{D}^{-}_{qc}(X)$
			& $\mathbf{D}^{-}_{qc}(Y)$ \\
			
			$\mathcal{B}^{+}$
			& $\mathbf{D}^{+}_{qc}(X)$
			& $\mathbf{D}^{+}_{qc}(Y)$ \\
			
			$\mathcal{B}^{b}$
			& $\mathbf{D}^{b}_{qc}(X)$
			& $\mathbf{D}^{b}_{qc}(Y)$ \\
			
			$\mathcal{B}^{c}$
			& $\mathbf{D}^{\mathrm{perf}}(X)$
			& $\mathbf{D}^{\mathrm{perf}}(Y)$ \\
			
			$\mathcal{B}^{-}_{c}$
			& $\mathbf{D}^{p}_{qc}(X)$
			& $\mathbf{D}^{p}_{qc}(Y)$ \\
			
			$\mathcal{B}^{b}_{c}$
			& $\mathbf{D}^{p,b}_{qc}(X)$
			& $\mathbf{D}^{p,b}_{qc}(Y)$ \\
			\hline
		\end{tabular}
	\end{center}
	
    Here, $\mathbf{D}^{p}_{qc}(X)$ denotes the subcategory of pseudo-coherent complexes, $\mathbf{D}^{p,b}_{qc}(X)$ denotes the subcategory of pseudo-coherent complexes with bounded cohomology.
	In each row, the category in the middle column
	is a separable extension of the category on the right,
	via the corresponding restriction of $\mathbf{T}$.
	
	The behavior of intrinsic subcategories under
	\'etale extensions is also used in
	\cite[Corollary~4.14]{KK} to study regularity
	for schemes proper over a commutative noetherian ring.
\end{exa}
\bigskip {\bf Acknowledgements}
\bigskip

This research was partially supported by the Natural Science Basic Research Program of Shaanxi Province (Grant No. 2024JC-YBMS-032).

\end{document}